\documentclass[a4paper,11pt]{amsart}
\pdfoutput=1   
\usepackage[utf8]{inputenc}
\usepackage[english]{babel}
\usepackage{indentfirst}
\usepackage{csquotes}
\usepackage{enumerate}

\usepackage{amsmath,amsthm,amssymb,mathrsfs,thmtools}
\usepackage{dsfont}

\usepackage{enumitem}
\usepackage[toc,page]{appendix}

\usepackage{graphicx}
\usepackage{mathtools,array}
\usepackage[all]{xy}
\usepackage{tikz-cd}

\usepackage{fancyhdr,floatpag}
\usepackage{upgreek,color}
\usepackage{hyperref}

\usepackage{cleveref}

\usepackage{blkarray}
\usepackage{multirow}

\usepackage{dutchcal} 

\usepackage[style=numeric,sorting=nyt,maxbibnames=7,maxcitenames=5,block=space,backend=bibtex]{biblatex}
\renewbibmacro{in:}{}

\hypersetup{linkbordercolor={0 0.60 0.75}, linkcolor=MidnightBlue}
\hypersetup{citebordercolor={0 0.60 0.75}, citecolor=MidnightBlue}

\numberwithin{equation}{section}
\theoremstyle{plain}

\newtheorem{thm}{Theorem}[section]
\newtheorem{conj}[thm]{Conjecture}

\newtheorem{cor}[thm]{Corollary}

\newtheorem{lem}[thm]{Lemma}

\newtheorem{pps}[thm]{Proposition}

\theoremstyle{definition}
\newtheorem{dfn}[thm]{Definition}
\newtheorem{prob}[thm]{Problem}
\newtheorem{rmkx}[thm]{Remark}
\newenvironment{rmk}
  {\pushQED{\qed}\renewcommand{\qedsymbol}{$\triangle$}\rmkx}
  {\popQED\endrmkx}
\newtheorem{examplex}[thm]{Example}
\newenvironment{exm}
  {\pushQED{\qed}\renewcommand{\qedsymbol}{$\triangle$}\examplex}
  {\popQED\endexamplex}

\declaretheoremstyle[
  spaceabove=-6pt,
  spacebelow=6pt,
  headfont=\normalfont\bfseries,
  postheadspace=1em,
  qed=\qedsymbol,
  headpunct={}
]{mystyle} 
\declaretheorem[name={proofof},style=mystyle,unnumbered,
]{proofof}  
\makeatletter
\renewenvironment{proofof}[1] {\par\pushQED{\qed}\normalfont\topsep6\p@\@plus6\p@\relax\trivlist  \item[\hskip\labelsep
        \bfseries
    Proof of #1.]\ignorespaces}{\popQED\endtrivlist\@endpefalse}
\makeatother
\makeatletter
\newcommand{\thickhline}{
    \noalign {\ifnum 0=`}\fi \hrule height 1pt
    \futurelet \reserved@a \@xhline
}
\newcolumntype{"}{@{\hskip\tabcolsep\vrule width 1pt\hskip\tabcolsep}}
\makeatother

\makeatletter
\renewenvironment{proof}[1][\proofname] {\par\pushQED{\qed}\normalfont\topsep6\p@\@plus6\p@\relax\trivlist\item[\hskip\labelsep\bfseries#1\@addpunct{.}]\ignorespaces}{\popQED\endtrivlist\@endpefalse}
\makeatother

\makeatletter
\renewcommand{\@secnumfont}{\bfseries}
\makeatother
\usepackage{etoolbox}
\makeatletter
\patchcmd{\section}{\scshape}{\bf}{}{}
\patchcmd{\subsection}{\scshape}{\bf}{}{}
\patchcmd{\subsubsection}{\scshape}{\bf}{}{}
\makeatother

\newcommand{\N}{\mathbb{N}} 
\newcommand{\Z}{\mathbb{Z}} 
\newcommand{\Q}{\mathbb{Q}} 
\newcommand{\R}{\mathbb{R}} 
\newcommand{\K}{\mathbb{K}} 
\newcommand{\F}{\mathbb{F}} 
\newcommand{\OS}{\mathcal{O}_S} 

\newcommand{\anel}[1]{#1_{\mathrm{ring}}} 

\DeclareMathOperator{\GL}{GL} 
\DeclareMathOperator{\SL}{SL} 
\newcommand{\eij}{e_{i,j}} 
\newcommand{\di}{d_i} 
\newcommand{\dk}[1]{d_{#1}} 
\newcommand{\ekl}[1]{e_{#1}} 
\newcommand{\bekl}[1]{\mathbf{e}_{#1}} 
\newcommand{\bd}[1]{\mathbf{d}_{#1}} 
\newcommand{\cekl}[1]{{\varepsilon}_{#1}} 
\newcommand{\cd}[1]{{\delta}_{#1}} 
\newcommand{\mbU}{\mathbf{U}} 
\newcommand{\mbD}{\mathbf{D}} 
\newcommand{\sbgpdi}{\mathcal{D}_{i}} 

\newcommand{\PGL}{\mathbb{P}\mathrm{GL}}

\newcommand{\Aff}{\mathbb{A}\mathrm{ff}} 
\newcommand{\mbA}{\mathbf{A}}
\newcommand{\ombA}{\overline{\mathbf{A}_4(R)}}
\newcommand{\ombU}{\overline{\mathbf{U}_4(R)}}
\newcommand{\mbB}{\mathbf{B}} 
\newcommand{\TI}{\mathbf{T}^I} 
\newcommand{\TIc}{\mathbf{T}^{I^c}} 
\newcommand{\PB}{\mathbb{P}\mathbf{B}} 
\newcommand{\mbS}{\mathbf{S}^I} 

\newcommand{\phee}{\varphi} 
\newcommand{\veps}{\varepsilon} 
\newcommand{\barra}[1]{\overline{#1}}

\newcommand{\gera}[1]{\langle {#1} \rangle}
\newcommand{\set}[1]{\{ #1 \}}
\newcommand{\vazio}{\varnothing} 

\newcommand{\mbf}{\mathbf}
\newcommand{\mb}{\mathbf}
\newcommand{\mc}{\mathcal}

\newcommand{\into}{\hookrightarrow}
\newcommand{\onto}{\twoheadrightarrow}
\newcommand{\nsgp}{\trianglelefteq} 

\newcommand{\Ri}{R_{\infty}} 

\newcommand{\Fn}[1]{\mathtt{F}_{#1}} 
\newcommand{\FPn}[1]{\mathtt{FP}_{#1}}
\newcommand{\vcd}{\mathtt{vcd}} 
\definecolor{amethyst}{rgb}{0.6, 0.4, 0.8}

\DeclareMathOperator{\Aut}{Aut}
\DeclareMathOperator{\Inn}{Inn}

\DeclareMathOperator{\carac}{char}
\DeclareMathOperator{\id}{id}
\DeclareMathOperator{\Fix}{Fix}
\DeclareMathOperator{\Sym}{Sym}

\title[Diversity of soluble groups with $R_\infty$]{Cohomological and quasi-isometric diversity of groups with property $R_\infty$}
\author{Karel Dekimpe, Paula M. Lins de Araujo and Yuri Santos Rego}
\address{Katholieke Universiteit Leuven, \newline 
Campus Kulak Kortrijk, \newline 
Etienne Sabbelaan 53,  \newline 
8500 Kortrijk, Belgi\"e}
\email{karel.dekimpe@kuleuven.be}
\address{University of Lincoln, \newline 
Charlotte Scott Research Centre for Algebra, College of Health and Science, \newline 
Isaac Newton Building, Brayford Pool, \newline 
LN6 7TS Lincoln, United Kingdom}
\email{pmacedolinsdearaujo@lincoln.ac.uk}
\email{ysantosrego@lincoln.ac.uk}

\begin{document}
\begin{abstract}
How rich is the collection of groups with a given prominent property? In this work we approach this question for property~$R_\infty$, which says that every automorphism $\varphi$ of a given group has infinitely many orbits under the $\varphi$-twisted conjugation action $(g,x) \mapsto gx\varphi(g)^{-1}$. Generalising the soluble groups of Herbert Abels to a large family over many integral domains, we prove that most such groups have property~$R_\infty$ drawing from a classical result of Levchuk and a swift observation by Jabara. Within the broad programme of cataloguing finitely generated groups up to quasi-isometry, our groups can then be separated by finiteness properties and cohomological dimension whilst having~$R_\infty$. Abandoning finite presentability, we establish that property~$R_\infty$ is very abundant in a strong sense: there are uncountably many finitely generated groups (which can all be chosen to be amenable or non-amenable) that have~$R_\infty$ and are pairwise not quasi-isometric. The proofs vary in flavour. On the amenable side we use carefully constructed quotients of Abels' groups and a general strategy for quasi-isometric diversity established by Minasyan, Osin, and Witzel. For the non-amenable constructions we rely on modifications of Leary's type $\mathtt{FP}$ groups, further cohomological arguments, and recent powerful criteria for~$R_\infty$ due to Iveson, Martino, Sgobbi, Wong, and Fournier-Facio.
\end{abstract}
\maketitle \vspace{-1.0cm}
\thispagestyle{empty}

\section{Introduction}

We begin recalling \emph{twisted conjugacy}, where one allows automorphisms to take part in usual conjugation. 
Let $G$ be a group and $\varphi\in \Aut(G)$. Two elements $g, \, h \in G$ are $\varphi$-twisted conjugate (or simply $\varphi$-conjugate) if $g = x h \varphi(x)^{-1}$ for some $x\in G$. The equivalence classes of this relation 
are called $\varphi$(-twisted) conjugacy classes, or \emph{Reidemeister classes} of $\varphi$. The number of $\varphi$-conjugacy classes found in $G$ is the \emph{Reidemeister number} $R(\varphi) \in \Z_{>0} \cup \{\infty\}$. A central question in the area, which we motivate below in \Cref{sec:motivation}, is whether a group $G$ has \emph{property $R_\infty$}; that is, whether $R(\varphi)=\infty$ for all $\varphi\in \Aut(G)$. 

This article addresses two aspects of property~$\Ri$ in the universe of finitely generated groups, while contributing to the classification of (finitely generated) elementary amenable groups with this property. Firstly, 
by exhibiting a family of soluble groups with~$R_\infty$ that can be separated by 
finiteness length and by cohomological dimension, 
we add to a growing literature of groups with prescribed features that are separated by cohomological finiteness properties \cite{DesiSaidSSFPn,RachelStefanMatt,MonikaExotic,ClaudioPyInventiones,ClaudioEduardXiaolei}. 
Secondly, we note that there are uncountably many quasi-isometry classes of (finitely generated) groups with property~$R_\infty$, both in the non-amenable and in the amenable realms. The non-amenable case is a non-trivial application of recent advances in the literature \cite{IanFP,IvesonMartinoSgobbiWong2025}. 
The amenable case uses an explicit construction of quotients of members of our elementary amenable family mentioned above.

The groups we look at --- with the exception of those in \Cref{sec:qidNonamenable} --- are in fact soluble, inspired by work of Herbert Abels \cite{Abels0,Abels}, so let us briefly introduce them. 
Suppose $R$ is an integral domain, let $n\geq 2$ be given, and take $I \subseteq \{1, \ldots, n\}$. This data gives rise to the group of all upper triangular matrices $\mbS_n(R) \leq \GL_n(R)$ whose $j$-th diagonal entry 
is fixed to be $1$ whenever $j \notin I$; cf. \Cref{sec:structure} for details.

Well-known examples in this family include the following: 
    Taking $I=\{1, \ldots, n\}$, we have 
    \[\mbS_n(R) = \mbB_n(R) = \left( \begin{smallmatrix} * & * & \cdots & * \\  & * & & \\ & & \ddots & \vdots \\ & & & * \\  \end{smallmatrix} \right) \leq \GL_n(R),\]
    the group of $R$-points of the standard Borel subgroup of $\GL_n$; 
If we choose $I$ so that $|I| = n-1$, then 
    \[\mbS_n(R) \cong \PB_n(R) \leq \PGL_n(R),\]
    the group of $R$-points of the standard Borel subgroup of the projective group $\PGL_n$; 
    and 
 when $I = \{2,\ldots,n-1\}$, we obtain 
    \[\mbS_n(R) = \mbA_n(R) = \left( \begin{smallmatrix}
  1 & * & \cdots & \cdots & * \\
  0 & * & \ddots & & \vdots \\
  \vdots & \ddots & \ddots & \ddots & \vdots \\
  0 & \cdots & 0 & * & * \\
  0 & \cdots & \cdots & 0 & 1
 \end{smallmatrix} \right) \leq \GL_n(R),\]
where $\mbA_n$ denotes the $\Z$-group schemes of H.~Abels; see~\cite[Appendix~A]{BenGriHar}, \cite[Section~1.2]{YuriSoluble}, and references therein for an overview on Abels' groups. 
Readers familiar with algebraic fibring and the $\Sigma$-invariants of Bieri--Neumann--Strebel \cite{BNS} will promptly recognise many of our groups $\mbS_n(R)$ as kernels of characters of $\mbB_n(R)$, i.e. homomorphisms from $\mbB_n(R)$ to $\R$.

Our first theorem uses the groups $\mbS_n(R)$ to distinguish (countably many) quasi‑isometry classes of soluble groups  with~$\Ri$ according to cohomological features. Recall that the \emph{finiteness length} $\phi(G)$ of a group $G$ is the smallest $n \geq 0$ for which $G$ has an Eilenberg--MacLane space with compact $n$-skeleton, if such $n$ exists, and otherwise $\phi(G)=\infty$. And we denote by $\vcd(G)$ the \emph{virtual cohomological dimension}, which is defined as the cohomological dimension of any torsion-free subgroup of finite index in $G$, if such a subgroup exists, and otherwise $\vcd(G)=\infty$.

\begin{thm}\label{cohomologicalmainthm}
The groups $\mbS_n(R)$ can be separated by finiteness properties and cohomological dimension while having property~$R_\infty$ and being finitely presented.

That is to say, for every $n\geq 4$ there exists an infinite subset $D_n \subseteq \N$ such that, given $d \in D_n \cup \{\infty\}$, then the subset $I \subseteq \{1,\ldots, n\}$ and the ring $R$ can be so chosen that:
\begin{enumerate}
\item $\vcd(\mbS_n(R)) = d$ and $\phi(\mbS_n(R)) = n-2$;
\item $\mbS_n(R)$ has property~$R_\infty$.
\end{enumerate}
\end{thm}

Our second result demonstrates that property~$R_\infty$ is in fact extremely abundant in the universe of finitely generated groups. 

\begin{thm}\label{diversemainthm}
In each of the classes of non-amenable and of amenable groups, there exist uncountably many, pairwise non-quasi-isometric, finitely generated groups with property~$R_\infty$. 
More precisely: 
\begin{enumerate}
\item \label{diversemainthm1} In the former class, there are uncountably many, pairwise non-quasi-isometric, accessible, infinitely-ended, finitely generated groups with property~$R_\infty$ and of homological type $\FPn{}$;
\item \label{diversemainthm2} In the latter class, there is a Dedekind domain of arithmetic type $\OS$ in positive characteristic such that Abels' group $\mbf{S}_4^{\{2,3\}}(\OS) = \mbA_4(\OS)$ is finitely presented and has uncountably many pairwise non-quasi-isometric quotient groups all of which have property~$R_\infty$.
\end{enumerate}
\end{thm}

Experts in the field will find the statement in \Cref{diversemainthm}\eqref{diversemainthm1} plausible. Modulo the use of many other famous results, the construction relies on two main ingredients that have only recently appeared in the literature: 
on the~$R_\infty$ side it is an application of a recent breakthrough by Iveson--Martino--Sgobbi--Wong and Fournier-Facio \cite{IvesonMartinoSgobbiWong2025}; on the topological side, we use Leary's uncountable family of duality groups (see \cite[Section~18]{IanFP} and \cite{KLS}) in an adapted version of an argument due to Li--S\'anchez Salda\~{n}a \cite[Theorem~4.2]{KevinLuis0}. 
The cohomological and coarse geometric statements in Theorems~\ref{cohomologicalmainthm} and~\ref{diversemainthm}\eqref{diversemainthm2} 
follow from established cohomological and topological methods: finiteness properties of soluble (mostly $S$-arithmetic) linear groups \cite{Abels,Bux04,YuriSoluble}, quasi-isometric invariance of such properties and of cohomological dimension \cite{AlonsoFPnQI,RomanQI}, and quasi-isometric diversity within the space of marked finitely generated groups \cite{MinasyanOsinWitzel}. 
 The key missing step is then to prove that the groups under consideration have~$R_\infty$. In fact, we do a bit more, but before discussing this missing ingredient we need a definition.

\begin{dfn} \label{def:conditionI}
Given a natural number $n \geq 2$, we say that the set $I \subseteq \{1,\ldots,n\}$ satisfies the (NG)-condition (for ``no gaps'') when the following holds:
\begin{equation}\label{conditionI}
\forall i\in \{1, \dots, n-1\}, \quad i \notin I \Longrightarrow i + 1 \in I.
\end{equation}
\end{dfn}

Our main theorem is a partial classification of finitely generated groups $\mbS_n(R)$ with property~$R_\infty$. 

\begin{thm}\label{mainthm}
Let $R$ be an integral domain either of characteristic zero, or of the form 
$R=\F_q[t,f_1(t)^{-1}, \ldots, f_{\ell-1}(t)^{-1}]$, where $\ell \geq 2$ and the polynomials $f_i(t) \in \F_q[t]$ are irreducible and pairwise coprime. 
If $n \geq 4$ and $I \subseteq \{1,\ldots,n\}$ satisfies the (NG)-condition, then $\mbS_n(R)$ has property~$\Ri$ whenever $\mbS_n(R)$ is finitely generated. 
\end{thm}

While the (NG)-condition \eqref{conditionI} might seem artificial at first glance, it arises naturally: for many interesting domains $R$, the (NG)-condition is necessary for $\mbS_n(R)$ to be finitely generated; cf. \Cref{pps:SIfinitelygenerated}. 

Moreover, we remark that the hypotheses on the size $n\geq 4$ of our matrices and the `no gaps shape' of $I \subseteq \{1,\ldots,n\}$ are optimal in a certain sense. Precisely:

\begin{pps} \label{obs:mainthmisoptiomal}
Let $n \geq 2$ be an integer number.
    \begin{enumerate} 
    \item For all such $n$, one may always choose 
    an integral domain $R$ of characteristic zero 
    for which the groups $\mb{S}_n^{\varnothing}(R)$ do not have property~$R_\infty$.
    \item If $n < 4$, there exist 
    integral domains $R,$ either of characteristic zero or in $\F_q(t)$, 
    and a non-empty subset $I \subseteq \{1, \ldots, n\}$ satisfying the (NG)-condition such that $\mb{S}_n^I(R)$ does not have property~$R_\infty$. 
    \end{enumerate}
\end{pps}

We claim no originality in \Cref{obs:mainthmisoptiomal} as it can be puzzled together from known results in the literature, although we contribute with a few examples that seem to be new; see Section~\ref{sec:reallastsection} for a proof of \Cref{obs:mainthmisoptiomal}. 

\subsection{Motivation, context and open questions} \label{sec:motivation} 
Property~$R_\infty$ first rose to prominence due to its topological implications. Decades after Lefschetz, Nielsen, and Reidemeister developed classical fixed-point theory, a series of fundamental observations made by 
Franz \cite{Franz}, Jiang \cite{Jiang0}, 
Brooks--Brown--Pak--Taylor \cite{BBPT}, Anosov \cite{AnosovNil} and Fadell--Husseini \cite{FadellHusseini}, and Keppelmann--McCord \cite{KeppelmannMcCordSolv} yielded the following: for special types of compact connected solvmanifolds $M$, if $\pi_1(M)$ has property~$R_\infty$ then the Nielsen number of every self-homotopy equivalence of $M$ vanishes. Moreover, if $R_\infty$ is absent from $\pi_1(M)$, Nielsen numbers of self-homotopy equivalences of $M$ are controlled by Reidemeister numbers of automorphisms of $\pi_1(M)$. We refer the reader to \cite{JiangPrimer} for a modern exposition on fixed-point theory and \cite{FelshtynHillWong,KarelIrisInfrasolv} for the relationship between Reidemeister and Nielsen numbers in such spaces. This allowed for the construction of manifolds whose self-homeomorphisms have a prescribed number of essential fixed points, by starting with a fundamental group $\pi_1(M)$ that is e.g. free nilpotent \cite{DacibergWongCrelle,KarelDacibergNil,KarelSamVargas,DekimpeLathouwers}, a nilpotent quotient of a surface group \cite{KarelDacibergSurface}, or polycyclic of exponential growth \cite{KarelSamIrisSolv,Bn1}.

On the purely group-theoretic side, Zappa \cite{Zappa0} seems to have been the first to investigate the presence or absence of property~$R_\infty$, studying how this affects the structure of some polycyclic groups. There was rekindled interest on this property following up on works of Fel'shtyn--Hill \cite[Section~1.5.3]{FelshtynHill} and Levitt--Lustig \cite[Theorem~3.5]{LevittLustig}. Recalling that (in adequate models) `almost all' random groups are word-hyperbolic \cite{GromovEssays}, property~$R_\infty$ naturally abounds within the universe of finitely presented groups because word-hyperbolic groups have~$R_\infty$. This was first outlined by Levitt--Lustig \cite{LevittLustig} and Fel'shtyn \cite{FelshtynHyperbolic} relying, with some inaccuracies, on an observation due to Delzant \cite[Lemma~3.4]{LevittLustig}. A thorough account and detailed proof of Delzant's lemma was then given by Soroko--Vaskou in a recent preprint \cite[Section~5]{SorokoVaskouArtin}, thus fully clarifying the situation for word-hyperbolic groups. 
Consequently, there has been an increased interest in establishing results relating~$R_\infty$, or the lack thereof, to a group's structure or its actions; see, for instance, \cite[Theorem~C]{Jabara}, \cite[Corollary~1.20]{TimmXiaolei0}, \cite[Theorem~2.5]{BhuniaBose2}, \cite[Theorem~5.4]{Brasuca0}, and \cite[Corollary~8.1.4]{IvesonMartinoSgobbiWong2025} for some examples of various flavours.

But in the spirit of classifying all finitely generated groups up to quasi-isometry, how abundant is property~$R_\infty$? The historical discussion above indicates a clear separation: the literature splits into searching for~$R_\infty$ for amenable and for non-amenable groups. Regarding the latter case, there is a general belief that `most' finitely generated non-amenable groups will end up having property~$R_\infty$. Concretely, \cite[Conjecture~(NA)]{SorokoVaskouArtin} states that finitely generated, non-amenable, residually finite groups should have~$R_\infty$, and many examples in the literature support this conjecture. 

A recent breakthrough offers a representative illustration. Iveson, Martino, Sgobbi, and Wong \cite{IvesonMartinoSgobbiWong2025} heavily expanded on the ideas of Levitt--Lustig \cite{LevittLustig} (sustained by Delzant's lemma) and showed that all finitely generated accessible groups with infinitely many ends have property~$R_\infty$. Immediately thereafter, Fournier-Facio vastly generalised both this result and a cohomological criterion of Gon\c{c}alves--Kochloukova \cite{DesiDaciberg} as slightly improved in \cite[Theorem~5.4]{Brasuca0}. His more general criterion for~$R_\infty$ is proved with a short argument using quasi-morphisms, featured as an appendix to the same paper \cite{IvesonMartinoSgobbiWong2025}. As a consequence of their findings, there are indeed uncountably many quasi-isometry classes of non-amenable groups with property~$R_\infty$. As this statement has not appeared in the literature, we record it in the form of \Cref{diversemainthm}\eqref{diversemainthm1} here, briefly outlining in \Cref{sec:qidNonamenable} how one such collection of groups can be constructed to invoke \cite{IvesonMartinoSgobbiWong2025}.

In contrast, it is completely unclear which finitely generated amenable groups exhibit property~$R_\infty$. Since such groups are either virtually-$\Z$ or one-ended, they do not fit the framework of the paper~\cite{IvesonMartinoSgobbiWong2025}. Moreover, as mentioned above, the subclass of (vritually) polycyclic groups has played a central role and have attracted considerable attention since these are the ones leading to constructions of interesting manifolds with prescribed Nielsen numbers {for all self-homotopy equivalences}. To obtain such manifolds, a number of authors discovered nilpotent or polycyclic groups with, and many without, property~$R_\infty$, albeit with no clear universal pattern showing `what makes' these groups have~$R_\infty$ or not; see \cite{DacibergWongCrelle,KarelSamVargas,KarelDacibergNil,KarelDacibergSurface,KarelSamIrisSolv,KarelIrisInfrasolv,Bn1,DekimpeLathouwers,Bn2} and references therein for a non-exhaustive list.

The known sources of (finitely generated) nilpotent and polycyclic groups with property~$R_\infty$ display a weak form of `quasi-isometric diversity', as they can be distinguished by cohomological dimension. Though much like hyperbolic groups, a further form of cohomological distinction amongst such groups is not so immediate since they are all of homotopical type $\Fn{\infty}$. In particular, they are finitely presented and thus only countably many quasi-isometry classes can be found there. 

These considerations motivated the main results of the present paper. 
Specifically, we wanted to address the following natural questions: 
\begin{itemize}
    \item Are there (finitely generated) amenable groups with property $R_\infty$ separated by finiteness properties? 
    \item Are there (finitely generated) amenable non-polycyclic groups with property $R_\infty$ separated by cohomological dimension? 
    \item Are there uncountably many quasi-isometry classes of (finitely generated) groups with property $R_\infty$?
\end{itemize}

Theorems~\ref{cohomologicalmainthm} and~\ref{diversemainthm} thus answer all of the above in the affirmative. 
Once settled, these questions lead one to wonder about the following broad problem, taking inspiration from a recent preprint of Vankov \cite{VladUncountableQI}.

\begin{prob}
Given a quasi-isometry class of a group with property~$R_\infty$, describe which groups in this class (up to isomorphism) also have~$R_\infty$.
\end{prob}

Our key technical contribution, \Cref{mainthm}, leaves open the question whether arbitrary (infinite) integral domains $R$ in positive characteristic can be plugged in to get $\mbS_n(R)$ with property~$R_\infty$. Indeed, there is no impediment coming from Levchuk's theorem, and the other core idea used --- Jabara's trick --- leads us to search for infinitely many fixed points in the base ring $R$ for every $\alpha \in \anel{\Aut}(R)$. Since such an $R$ in positive characteristic will have to be finitely generated as a module over its units (due to \Cref{pps:SIfinitelygenerated}), it seems reasonable that these domains will behave similarly to those in the current statement of \Cref{mainthm}. We thus pose:

\begin{conj}\label{prob3}
Let $n \geq 4$ and let $I\subseteq \{1, \dots, n\}$ be a set satisfying the (NG)-condition~\eqref{conditionI}. Suppose $R$ is an infinite integral domain of \emph{arbitrary} characteristic and such that $\mbS_n(R)$ is finitely generated. Then $\mbS_n(R)$ has property~$\Ri$.
\end{conj}


\subsection{Outline of the paper} 
We open with \Cref{sec:structure}, which provides the background needed for the remainder of the paper, introduces several auxiliary results, and presents useful properties and detailed descriptions of the groups we investigate. 

\Cref{sec:3} illustrates our approach by establishing \Cref{mainthm} in the particular case of $\PB_n(R)$. This serves as a model for the general argument, which is carried out in \Cref{sec:4}.

\Cref{sec:separating} develops the complete form of \Cref{cohomologicalmainthm}, refining the formulation given in the introduction. This result appears there in the form of two theorems, both consequences of \Cref{mainthm}. 

In~\Cref{sec:Abels}, we deal with quotients of Abels' groups $\mbA_4(\OS)$. More precisely, in \Cref{sec:notpenultimate}, we show that there is a Dedekind domain $\OS = R_f$ such that a quotient of $\mbA_4(R_f)$ by a characteristic subgroup has property~$\Ri$. 
The previous \Cref{sec:fundamentallemma} provides the ground for such proof.
We proceed in \Cref{sec:qidAn} with a proof of \Cref{diversemainthm}\eqref{diversemainthm2}, 
which follows from the fact that the mentioned quotients of $\mbA_4(R_f)$ have~$\Ri$. Then in \Cref{sec:qidNonamenable} we give one construction of (uncountably many, non-quasi-isometric) groups that fit into the framework of~\cite{IvesonMartinoSgobbiWong2025}, demonstrating that quasi-isometric diversity combined with~$R_\infty$ is also present in the class of non-amenable groups. 

We close with the proof of \Cref{obs:mainthmisoptiomal} in \Cref{sec:reallastsection}, showing that the bound $n \geq 4$ in \Cref{mainthm} is sharp and the `shape' of $I \subseteq \{1,\ldots,n\}$, as determined by the (NG)-condition, is optimal.


\section{Background}\label{sec:structure}
In this section, we collect the background needed for the remainder of the paper.
We begin by presenting a few auxiliary results on twisted conjugacy, and by introducing notation and conventions used throughout. 
We turn to the groups $\mbS_n(R)$, discussing useful properties including generating sets and relations. 
We end the preliminaries discussing characteristic subgroups.  

\subsection{Setup, notation, and auxiliary results}\label{sec:conventions}

All rings that appear in this paper are assumed to be unital with $1 \neq 0$, and we denote the group of units of a ring~$R$ by $U(R)$. 

For any group $G$, we sometimes use the notation 
\[\mu(g,h) := h g h^{-1}\] for the conjugation action, as we will often conjugate by, and with, expressions of considerable length.  Moreover, for fixed $g\in G$ we also denote by $\iota_g \in \Aut(G)$ the induced inner automorphism 
\[\iota_g(h) = ghg^{-1}.\] 
Our commutator convention is $[x,y]= xyx^{-1}y^{-1}$. The commutator subgroup of a group $G$ is interchangeably denoted by $G'$ and by $[G,G]$.

With our notation fixed, we brielfy discuss twisted conjugacy. 
The following 
results are all standard, but we include them here to keep the article as self‑contained as possible.

\begin{lem}[{E.g., \cite[Cor.~2.5]{FelshtynTroitskyCrelle}}]\label{lem:ignoreinner}
Let $G$ be a group and $\phee \in \Aut(G)$. Then the Reidemeister number of $\varphi$ is invariant under composition with inner automorphisms. That is,
\begin{equation}\label{eq:ReidInn} R(\iota_g \circ \phee) = R(\phee \circ \iota_g) = R(\phee), \quad \text{ for all } g \in G.
\end{equation}
\end{lem}

\begin{lem}[{E.g., \cite[Lemma~1.1]{DacibergWongCrelle}}]\label{lem:quotient}
Let $G$ be a group with normal subgroup~$N$. Assume that $\phee \in \Aut(G)$ is such that $\phee(N) =N$. Then $\phee$ induces an automorphism $\barra{\phee}\in \Aut(G/N)$ and, if $R(\barra{\phee}) =\infty$, then also $R(\phee)=\infty$.
\end{lem}

As a somewhat opposite result to the previous one, \cite[Lemma~2.9]{Heath} implies the following result, which allows us to show in certain situations that $R(\varphi)<\infty$ by considering a $\varphi$--invariant normal subgroup $N$.

\begin{lem}\label{lem:heath}
Let $G$ be a group with normal subgroup $N$. Assume that $\phee \in \Aut(G)$ is such that $\phee(N) =N$. Then $\phee$ induces an automorphism $\barra{\phee}\in \Aut(G/N)$. Denote by $\iota_g \circ \phee'\in \Aut(N)$ the restriction of $\iota_g \circ \phee$ to $N$ for every $g\in G$. If $R(\barra{\phee}) <\infty$ and also 
$R(\iota_g \circ \phee')<\infty$ for all $g \in G$, then $R(\varphi)< \infty$ as well.
\end{lem}

For a group $G$ and $\phee\in\Aut(G)$, we denote by $\Fix(\phee)$ the set of fixed points of $G$ under~$\phee$, that is, 
\[\Fix(\phee)=\{g\in G \mid \phee(g)=g\}.\]
We shall need the following crucial observation.

\begin{lem}[{Jabara's trick \cite{Jabara}}] \label{lem:Jabara}
Let $G$ be a finitely generated, residually finite group, and let $\phee \in \Aut(G)$. If $|\Fix(\phee)|=\infty$,  then $R(\phee) = \infty$.
\end{lem}

A detailed proof of Jabara's lemma can be found in \cite[Proposition~3.7]{PieterProducts}. 

In the case of nilpotent groups, we have the following converse to Jabara's result.

\begin{lem} \label{inverse-Jabara}
Let $N$ be a finitely generated torsion-free nilpotent group and $\varphi\in \Aut(N)$. Then $R(\varphi)< \infty$ if, and only if, $ \Fix(\varphi)=\{1\}$.
\end{lem}
\begin{proof}
If $\Fix(\varphi)\neq \{1\}$, then $\Fix(\varphi)$ is infinite as $N$ is torsion-free and $R(\varphi) = \infty$ by Jabara's trick. On the other hand, if $R(\varphi)=\infty$, then $\Fix(\varphi)$ is infinite by \cite[Proposition 3.14]{SendenProduct} --- note that $\Fix(\phee) = \mathrm{Stab}_{\phee}(1)$ in Senden's notation.
\end{proof}

\subsection{Finiteness properties, dimension, and quasi-isometries} \label{sec:geometry}

Our main applications concern (coarse) geometric properties of our groups of interest, and we now recall some of these concepts and key results. Although these are classic, we spell them out for the reader's convenience, 
following known sources such as \cite{SerreCohomologie,Bieri,BrownCohomology,GeoBook}. 

We start with \emph{finiteness properties} (sometimes also known as \emph{finiteness conditions}) and the finiteness length. Informally speaking, the finiteness length $\phi(G)$ of a group $G$ can be thought of as a measure of `how compact' $G$ can be from a topological perspective, up to homotopy --- if one interprets classifying spaces as the topological way of depicting $G$. 

We say that a group $G$ has (homotopical) \emph{type $\Fn{n}$} if it admits a classifying space $K(G,1)$ whose $n$-skeleton is finite. These generalise familiar properties because type $\Fn{1}$ (resp. type $\Fn{2}$) is equivalent to being finitely generated (resp. finitely presented); cf. \cite[Proposition~7.2.1]{GeoBook}. Note that these properties are increasingly stronger since type $\Fn{n}$ implies type $\Fn{k}$ when $k \leq n$. In case $G$ is of type $\Fn{n}$ for all $n \geq 0$, we say that $G$ is of type $\Fn{\infty}$. The (homotopical) \emph{finiteness length} of $G$, denoted by $\phi(G)$, is then defined as
\[\phi(G) = \sup\{ n \in \Z_{\geq 0} \mid G \text{ is of type } \Fn{n}\}.\]
In particular, $\phi(G) = \infty$ if and only if $G$ is of type $\Fn{\infty}$, and $\phi(G) \geq 2$ if and only if $G$ is finitely presented. Of course, $\phi(G)=n <\infty$ if and only if $G$ is of type $\Fn{n}$ but not of type $\Fn{n+1}$.

A useful fact is that the finiteness length behaves well under group extensions.

\begin{lem}[{E.g., \cite[Lemma~2.1]{YuriSoluble}}] \label{lem:finpropextensions}
Let $N \into G \onto Q$ be a short exact sequence of groups. Then:
\begin{enumerate}
\item $\phi(N) \geq n \leq \phi(Q) \implies \phi(G) \geq n$.
\item If the sequence splits (i.e., $G \cong N \rtimes Q$), then $\phi(G) \leq \phi(Q)$.
\end{enumerate}
\end{lem}

The finiteness length $\phi(G)$ can thus affect the structure of the (co)homology of $G$ since $H^\ast(G,-)$ and $H_\ast(G,-)$ may be computed out of a $K(G,1)$ space by building a projective resolution for the trivial $\Z[G]$-module $\Z$ out of a cell structure for $K(G,1)$; cf. \cite[Proposition~13.1.1]{GeoBook}. A related, alternative finiteness condition uses the length of such resolutions instead (\cite[Section~8.2]{GeoBook}). We say that $G$ has \emph{cohomological dimension} $n \in \Z_{\geq 0} \cup \{\infty\}$, denoted $\mathrm{cd}(G) = n$, when 
\[n = \sup\{d \in \Z_{\geq 0} \mid H^d(G,M) \neq 0 \text{ for some } \Z[G]\text{-module } M\}.\] 
Since $\mathrm{cd}(G) < \infty$ forces $G$ to be torsion-free \cite[p.~90]{SerreCohomologie}, a slight refinement of dimension is commonly used. A group $G$ is said to have \emph{virtual cohomological dimension} $\vcd(G) = n \in \Z_{\geq 0} \cup \{\infty\}$ when $G$ admits some torsion-free subgroup $H$ of finite index such that $\mathrm{cd}(H)=n$. (This concept does not depend on the choice of $H$; see \cite[Th\'eor\`eme~1, p.~96]{SerreCohomologie}. Here, the trivial group is considered torsion-free so that $\vcd(G)=0$ if and only if $G$ is finite.) In case no such subgroups exist, one has $\vcd(G) = \infty$.  

The virtual cohomological dimension behaves well with respect to subgroups. 

\begin{lem}[{E.g., \cite[p.~99]{SerreCohomologie}}] \label{lem:vcdsubgroups}
For a group $G$ and a subgroup $H \leq G$, one has $\vcd(H) \leq \vcd(G)$, with equality in case $H$ has finite index in $G$.
\end{lem}

Besides the above cohomological implications, the finiteness length $\phi(G)$ is also prominent for being a quasi-isometry invariant \cite{AlonsoFPnQI}. Recall that two metric spaces $(X, d_1)$ and $(Y, d_2)$ are \emph{quasi-isometric} to one another if there exist positive constants $L$, $C \in \R_{>0}$, a non-negative constant $D \in \R_{\geq 0}$, and a function $f : X \to Y$ such that 
\begin{itemize}
    \item $\frac{1}{L}d_1(p,q) - C \leq d_2(f(p),f(q)) \leq Ld_1(p,q)+C$ for all $p,q \in X$, and 
    \item for each $y \in Y$, there is some $x \in X$ with $d_2(f(x),y) \leq D$.
\end{itemize}
Being quasi-isometric is an equivalence relation. 
Recalling common practice in geometric group theory, a finitely generated group $G$ is viewed as a metric space by choosing a Cayley graph for it with the word metric. (Different generating sets yield quasi-isometric Cayley graphs.) 

The reader not so acquainted with property~$R_\infty$ should be warned that this property in general does not behave well under quasi-isometries --- for instance, for $\Z$ we have $R(-\id)=2$, whereas the virtually-$\Z$ group $D_\infty \cong C_2\ast C_2$ actually has~$R_\infty$; cf. \cite[Proposition~2.3]{DacibergWongCrelle}.
For finiteness properties, however, we have the following.

\begin{lem}[{E.g. \cite[Theorem~18.2.12]{GeoBook}}] \label{lem:FinLenQI}
If $G$ and $H$ are finitely generated quasi-isometric groups, then $\phi(G) = \phi(H)$.
\end{lem}

Likewise, the virtual cohomological dimension can also be used to distinguish groups up to quasi-isometry.

\begin{lem}[{E.g., \cite[Theorem~1.2(i)]{RomanQI}}] \label{lem:vcdQI}
Let $G$ and $H$ be countable groups. If they are quasi-isometric with both $\vcd(G)$ and $\vcd(H)$ finite, then $\vcd(G) = \vcd(H)$. In case $\vcd(G)=\infty$ but $\vcd(H)<\infty$, then $G$ and $H$ cannot be quasi-isometric. 
\end{lem}

\begin{exm} \label{ex:AbelianFinfty}
Recalling that the $n$-torus $\prod_{j=1}^n \mathbb{S}^1$ is a classifying space for the free abelian group $\Z^n$, we have that $\vcd(A)=n$ and $\phi(A) = \infty$ for any finitely generated abelian group $A$ of torsion-free rank $n$. Inducting on the nilpotency class (resp. on the Hirsch length) and using \Cref{lem:finpropextensions}, one readily checks that a finitely generated nilpotent (resp. polycyclic) group $G$ always has $\phi(G)=\infty$.
\end{exm}

Although we mostly use the finiteness length throughout the paper, we stress that the homotopical properties $\Fn{n}$ have homological counterparts which shall appear sporadically here. A group $G$ is of \emph{type $\FPn{n}$} when $\Z$, viewed as a $\Z[G]$-module with trivial $G$-action, has a projective resolution 
\[\cdots \to P_i \to P_{i-1} \to \cdots \to P_1 \to P_0 \to \Z\]
by $\Z[G]$-modules $\{P_i\}_{i \geq 0}$ whose terms $P_i$ are finitely generated at least up to index $n$. As in the homotopical case, $G$ is of \emph{type $\FPn{\infty}$} if it is of type $\FPn{n}$ for all $n \geq 0$. These homological properties are also quasi-isometric invariants \cite{AlonsoFPnQI}. A further refined finiteness condition we shall encounter is as follows: $G$ is said to be of \emph{type $\FPn{}$} when the trivial $\Z[G]$-module $\Z$ has a projective resolution $\{P_i\}_{i \geq 0}$ all of whose modules $P_i$ are finitely generated and that terminates at some finite index $n\geq 0$ (i.e., $P_j \cong 0$ for all $j \geq n$). Taking dimension into account, one has the following.

\begin{lem}[{E.g., \cite[Chapter~VIII, (6.1) Proposition]{BrownCohomology}}] \label{lem:FP}
A group $G$ is of type $\FPn{}$ if and only if both $\mathrm{cd}(G)<\infty$ and $G$ is of type $\FPn{\infty}$.
\end{lem}

\subsection{The groups \texorpdfstring{$\mbS_n(R)$}{SIn(R)}}\label{sec:SI}
Throughout, let $R$ be an integral domain and $n\geq 2$.
Consider the group of upper triangular matrices 
\[ \mbB_n(R) = \left( \begin{smallmatrix} * & * & \cdots & * \\  & * & \cdots & * \\ & & \ddots & \vdots \\ & & & * \\  \end{smallmatrix} \right)\leq \GL_n(R), \] 
and the group of upper unitriangular matrices
\[ \mbU_n(R) = \left( \begin{smallmatrix} 1 &  * & * & \cdots & * \\  & 1 & *&  \cdots &* \\[-1mm] & &\ddots & \ddots &  \vdots \\  & & & 1 & * \\
& & & & 1 \\  \end{smallmatrix} \right) \leq \GL_n(R).\] 
For each $i \in \{1,2,\ldots, n\}$ we let $\sbgpdi(R) \leq \GL_n(R)$ denote the subgroup of diagonal matrices whose $i$-th entry is an element of the unit group $U(R)$ and all other diagonal entries equal one. Note that the subgroup $\mbD_n(R)$ of diagonal matrices in $\mbB_n(R)$ is given by 
\[ \mbD_n(R) = \left( \begin{smallmatrix} * & &  &  \\  & * & & \\[-1mm] & & \ddots &  \\ & & & * \\  \end{smallmatrix} \right) = \mc{D}_1(R) \cdot \mc{D}_2(R) \cdot \ldots \cdot \mc{D}_n(R).\] 
An important quotient of $\mbB_n(R)$ is its projective version. 
As we only work with integral domains, the set $Z_n(R)$ of all multiples of the identity matrix $\mb{1}_n$, i.e. 
\[Z_n(R) = \set{u\cdot \mb{1}_n \in \GL_n(R) \mid u \in U(R)} \subseteq \mbD_n(R),\]
is always a central subgroup of $\mbB_n(R)$. 
The \emph{projective upper triangular group} $\PB_n(R)$ is then defined as the quotient 
\[\PB_n(R) = \frac{\mbB_n(R)}{Z_n(R)} \cong \mbU_n(R) \rtimes \frac{\mbD_n(R)}{Z_n(R)}.\]
Note that, if the group of units $U(R)$ of $R$ has only the trivial unit $1 \in R$, then $\PB_n(R) = \mbB_n(R) = \mbU_n(R)$.

We modify the diagonal of $\mbB_n(R)$ to construct families of groups that answer the questions that motivated our work. Given $n \in \N$, let $I \subseteq \{1, \dots, n\}$.  
We then denote by $\mbS_n(R)$ the subgroup of $\GL_n(R)$ composed of all upper triangular matrices with the following property: only the diagonal entries indexed by elements of $I$ can vary in $U(R)$, all others are fixed to be one. 

Put differently, consider the diagonal subgroup 
\[\TI_n(R)=\prod_{i \in I} \sbgpdi(R)\leq \mbD_n(R),\]
that is, $\TI_n(R)$ is the group of diagonal matrices where only the diagonal entries indexed by elements of $I$ can vary in $U(R)$, and all diagonal entries index by $I^c$ are fixed to be one. 
We then have that 
\[\mbS_n(R)=\mbU_n(R)\rtimes \TI_n(R).\]

By definition, $\mathbf{S}^{\{1, \dots, n\}}_n(R)=\mbB_n(R)$, whereas $\mbf{S}_n^{\vazio}(R) = \mbU_n(R)$. In turn, if $I$ is a proper subset of $\set{1,\ldots,n}$, then $\mbS_n(R)$ can be canonically identified with a 
subgroup of $\PB_n(R)$. Indeed, writing $[d]$ for the image in $\PB_n(R)$ of a diagonal matrix $d \in \mbB_n(R)$, we see that two elements $d_1, d_2 \in \TI_n(R)$ represent the same class $[d_1] = [d_2]$ in $\PB_n(R)$ if and only if $d_1d_2^{-1} = u\cdot \mb{1}_n$ for some $u \in U(R)$. Since $I \subsetneq \set{1,\ldots,n}$, at least one diagonal entry of $d_1d_2^{-1}$ equals $1$, whence $u = 1$ and so $d_1 = d_2$. Thus, the canonical map $\mbf{u}d \mapsto \mbf{u}[d]$ sending $\mbf{u} \in \mbU_n(R)$, $d \in \TI_n(R)$ to their images in $\mbU_n(R) \rtimes \left( \mbD_n(R)/Z_n(R)\right) \cong \PB_n(R)$ is an isomorphism onto its image. As a by-product of this argument, we obtain the following. 

\begin{lem} \label{lem:isowithPBn}
If the subset $I \subset \{1,\ldots,n\}$ has exactly $n-1$ elements, then $\mbS_n(R) \cong \PB_n(R)$.
\end{lem}

We briefly remark that the family $\{ \mbS_n \mid I \subseteq \{1,\ldots,n\}\}$ could have been defined and interpreted in the language of group schemes (in our case, affine and defined over $\Z$). The schemes $\mbS_n$ can be thought of as the $\Z$-group subschemes of $\GL_n$ that interpolate between the standard Borel subgroup $\mbB_n \leq \GL_n$ and its unipotent radical $\mbU_n \leq \mbB_n$. At the level of $R$-rational points, the family 
\[\{ \mbS_n(R) \mid n \geq 2, \, I \subseteq \{1,\ldots,n\}\}\] 
effectively generalises $\mbU_n(R)$, $\mbB_n(R)$ and $\PB_n(R)$ simultaneously, but also the group $\Aff(R) \cong \left( \begin{smallmatrix} \ast & \ast \\ 0 & 1 \end{smallmatrix} \right) \leq \GL_2(R)$ of affine transformations of $R$ (since $\Aff(R) \cong \PB_2(R)$; cf. \cite[Lemma~3.6]{Bn1}).

\subsection{Useful properties of \texorpdfstring{$\PB_n(R)$}{PBn(R)} and \texorpdfstring{$\mbS_n(R)$}{SIn(R)}} \label{sec:allrelations}
There is a well-known series of relations among diagonal and unipotent matrices in the general linear group; see e.g.,\cite{Silvester}. We recall some of them below.

Recall that $\mbD_n(R)$ is the subgroup of $\GL_n(R)$ composed of diagonal matrices, and we let $d_i(u)$ denote the diagonal matrix where the $i$-th entry is equal to $u\in U(R)$ and has ones elsewhere on the diagonal. (In particular, $\sbgpdi(R) = \{ d_i(u) \mid u \in U(R)\}$.)  
As we only work over integral domains, the diagonal subgroup is abelian and we thus have relations on $\mbD_n(R)$ and on $\TI_n(R)$ given by 
\begin{align}\label{rel:commutatorsD}
\begin{split}
\di(u) \di(v) & = \di(uv), \\
\di(u)\dk{j}(v) & = d_j(v) \di(u). 
\end{split}
\end{align}
In the above, if we consider $d_i(u)$ as an element of $\TI_n(R)$, we obviously assume $u=1$ whenever $i\notin I$.

Denote by $e_{i,j}(r)$ the matrix of $\mbU_n(R)$ whose off-diagonal $(i,j)$-entry is $r\in R$ and all other off-diagonal entries are zero. The relations in $\mbU_n(R)$ are
\begin{align}\label{rel:commutatorsU}
\begin{split}
\eij(r)\eij(s) &= \eij(r+s), \\
[\eij(r),\ekl{k,l}(s)]^{-1} &= [\eij(r),\ekl{k,l}(s)^{-1}], \text{ and } \\
[\eij(r),\ekl{k,l}(s)] & = 
\begin{cases}
\ekl{i,l}(rs) & \mbox{if } j=k,\\
1 & \mbox{if } i \neq l \text{ and } k \neq j.
\end{cases}
\end{split}
\end{align}
The groups $\mbD_n(R)$ and $\TI_n(R)$ act on $\mbU_n(R)$ as follows.
\begin{align}\label{rel:commutators}\di(u) \ekl{k,l}(r) \di(u)^{-1} & = 
\begin{cases}
\ekl{k,l}(ur) & \mbox{if } i=k,\\
\ekl{k,l}(u^{-1}r) & \mbox{if } i=l, \\
{\ekl{k,l}(r)} & \mbox{otherwise}.
\end{cases}\end{align}
In particular, if $d = d_1(u_1) \cdots d_n(u_n) $ is an element of $\mbD_n(R)$ or of $\TI_n(R)$, one has 
\begin{equation}\label{conjdiag} d \eij(r) d^{-1} = \eij(u_i u_j^{-1} r). \end{equation}
Relations entirely analogous to the above obviously hold in $\PB_n(R) \cong \mbU_n(R) \rtimes \frac{\mbD_n(R)}{Z_n(R)}$, where we write (representatives of) diagonal matrices $d_i(u)$ as $[d_i(u)]$ instead.

In the present paper we are mainly concerned with {countable} groups. It is precisely the condition of being finitely generated that will impose a combinatorial constraint on the diagonal part $\TI_n(R)$ of our groups $\mbS_n(R) = \mbU_n(R) \rtimes \TI_n(R)$. 
We thus record the following, whose proof is adapted from \cite[Proof of Theorem~1.2 (Items~1 \&~2)]{YuriSoluble}.

\begin{pps}\label{pps:SIfinitelygenerated} Let $R$ be a finitely generated commutative ring with unity and let $I \subseteq \{1,...,n\}$. Then the matrix group $\mbS_n(R)$ is finitely generated if, and only if, one (possibly both) of the following conditions is satisfied.
\begin{enumerate}
\item The underlying additive group $(R,+)$ is finitely generated. 
\item The multiplicative group $(U(R),\cdot)$ is finitely generated, the additive group $(R,+)$ is finitely generated as a module over the units $U(R)$, and furthermore the set $I$ satisfies the (NG)-condition~\eqref{conditionI}.
\end{enumerate}
\end{pps}

\begin{proof}
Recall that $\mbS_n(R)$ is a split extension $\mbS_n(R) = \mbU_n(R) \rtimes \TI_n(R)$. Now, by the relations described in \Cref{sec:allrelations}, the nilpotent subgroup $\mbU_n(R) \leq \mbS_n(R)$ is generated by elementary matrices in the positions $(i,i+1)$ with $i$ ranging through $\{1,\ldots,n-1\}$, while the diagonal subgroup $\TI_n(R)$ is generated by diagonal matrices whose only non-one entries are in position $j$, with $j$ running over the given index set $I$. That is,
\[\mbS_n(R) = \gera{d_j(u),e_{i,i+1}(r) \mid j \in I, i \in \{1,\ldots,n-1\}, u \in U(R), r \in R}.\] 
Let us first show that the above conditions~(i) and~(ii) are necessary. We start by assuming that $(R,+)$ is \emph{not} finitely generated. We then have to consider different cases from Condition~(ii). 

To begin with, if $U(R)$ is not finitely generated, then $\mbS_n(R)$ cannot be finitely generated since $U(R)$ is a homomorphic image of $\mbS_n(R)$ --- e.g., by mapping onto any diagonal subgroup $\sbgpdi(R)$ with $i \in I$.
{Note that $I\neq \vazio$, because $\mbf{S}_n^{\vazio}(R)=\mbU_n(R)$ is not finitely generated if $(R,+)$ is not finitely generated.} So let us assume that $U(R)$ is indeed finitely generated. As our rings are commutative, $U(R)$ becomes a finitely generated abelian group, hence $U(R)$ --- and therefore also the diagonal subgroup $\TI_n(R)$ --- is finitely presented. But we have the short exact sequence $\mbU_n(R) \into \mbS_n(R) \onto \TI_n(R)$ with a finitely presented quotient $\TI_n(R)$. It thus follows that, for $\mbS_n(R)$ to be finitely generated, then $\mbU_n(R)$ must be finitely generated as the normal subgroup of this sequence; see, for instance, \cite[2.1 Satz(b)]{AbelsC2}. In turn, using the Relations~\eqref{conjdiag} and looking at the quotient $\mbS_n(R) / \mbU_n(R)'$, we see that $\mbU_n(R)$ (and also $\mbU_n(R)/\mbU_n(R)'$) can only be finitely generated as a normal subgroup when $R$ is finitely generated as a module over its units. 
Lastly, still supposing that $(R,+)$ is \emph{not} finitely generated, assume $I$ does not satisfy the (NG)-condition~\eqref{conditionI}. 
That is, assume there exists $i \in \{1,\ldots,n-1\}$ such that the consecutive indices $i$, $i+1$ do not belong to~$I$.
In this case, all diagonal matrices in $\mbS_n(R)$ act trivially on the elementary matrices $e_{i,i+1}(r)$, $r \in R$. 
In particular, the map $\mbS_n(R)\to R: \mbf{x} \mapsto x_{i,i+1}$ that sends each matrix $\mbf{x} = (x_{i,j})$ of  $\mbS_n(R)$  to its $(i,i+1)$-th entry is a morphism onto $(R,+)$, whence $\mbS_n(R)$ cannot be finitely generated.

We now argue that the conditions in the statement are sufficient. First of all, if $(R,+)$ is finitely generated, then the nilpotent subgroup $\mbU_n(R)$ is itself finitely generated. Furthermore, the group of units $U(R)$ is also finitely generated by Samuel's strengthening of the Unit Theorem of Dirichlet; see~\cite[Section~4.7]{Samuel}. Thus the diagonal subgroup $\TI_n(R)$ is also finitely generated, hence so is the extension $\mbS_n(R) = \mbU_n(R) \rtimes \TI_n(R)$.

Suppose that the second condition is fulfilled, i.e., $U(R)$ is finitely generated, the ring $R$ is finitely generated as an $U(R)$-module and the (NG)-condition \eqref{conditionI} holds. The first two requirements imply that the subgroups {of $\mbB_n(R)$} of the form 
\[\gera{d_{i}(u), e_{i,i+1}(r) \mid u\in U(R), r \in R} \cong \left( \begin{smallmatrix} * & * \\ 0 & 1 \end{smallmatrix} \right) \leq \GL_2(R)\]
and
\[\gera{d_{i+1}(u), e_{i,i+1}(r) \mid u\in U(R), r \in R} \cong \left( \begin{smallmatrix} 1 & * \\ 0 & * \end{smallmatrix} \right) \leq \GL_2(R)\]
are finitely generated. But given an arbitrary index $i \in \{1,\ldots,n-1\}$, the (NG)-condition \eqref{conditionI} implies that $i$ or $i+1$ (possibly both) belongs to the index set $I$, whence for each $i \in \{1,\ldots,n-1\}$ one of the above subgroups is contained in $\mbS_n(R)$. One can thus construct from the above (finitely many) subgroups a finite generating set for the whole $\mbS_n(R)$, as desired.
\end{proof}

We shall also make repeated use of the following fact. 

\begin{lem}\label{lem:metabelianresiduallyfinite}
Let $R$ be a commutative ring and 
suppose  $I \subseteq \{1,\ldots,n\}$. If $\mbS_n(R)$ is finitely generated, then its quotient $\mbS_n(R)/\mbU_n(R)'$ by the commutator subgroup $\mbU_n(R)' = [\mbU_n(R),\mbU_n(R)]$ of the nilpotent subgroup $\mbU_n(R) \nsgp \mbS_n(R)$ is residually finite.
\end{lem}
\begin{proof}
Because $\mbS_n(R)$ is finitely generated, so is the quotient $\mbS_n(R)/\mbU_n(R)'$. Since $\mbS_n(R)/\mbU_n(R)'$ is metabelian, the claim follows from a famous theorem due to P.~Hall~\cite{PHallResFin}.
\end{proof}

When dealing with Reidemeister numbers, characteristic subgroups are very useful. 
It is known that, if $R$ is an integral domain, then $\mbU_n(R)$ is a characteristic subgroup of $\PB_n(R)$; cf.\ \cite[Proposition~3.9]{Bn1}. However, it is not true in general that $\mbU_n(R)$ is a characteristic subgroup of $\mbB_n(R)$; see \cite[Example~3.8]{Bn1}. 

In the sequel, we show that $\mbU_n(R)$ is characteristic in $\mbS_n(R)$, whenever $I$ is a proper subset of $\{1,\dots, n\}$. 

\begin{pps}\label{pps:char} If $R$ is an integral domain and $I$ is a proper subset of $\{1, \dots, n\}$, then $\mbU_n(R)$ is 
the Hirsch--Plotkin radical of $\mbS_n(R)$. 

That is, $\mbU_n(R)$ is the unique maximal subgroup $\mbS_n(R)$ with respect to being normal and locally nilpotent. In particular, $\mbU_n(R)$ is characteristic in $\mbS_n(R)$.
\end{pps}
\begin{proof}
If $I = \vazio$ or if the only unit of $R$ is the trivial unit $1$, then $\mbU_n(R)$ coincides with $\mbS_n(R)$ and the result becomes trivial. We can thus assume that $I \neq \vazio$ and that there exists $u\in U(R)$ such that $u\neq 1$. 

The goal is showing that $\mbU_n(R)$ is the unique maximal subgroup of $\mbS_n(R)$ with respect to being locally nilpotent \emph{and} normal. For this, let $M$ be a locally nilpotent normal subgroup of $\mbS_n(R)$. If $M$ is not a subset of $\mbU_n(R)$, then $\barra{M}=\mbU_n(R)M$ is a locally nilpotent normal subgroup of $\mbS_n(R)$ which strictly contains $\mbU_n(R)$. Hence 
$\barra{M}$ must contain a non-trivial element of the form 
\[d=\prod_{i\in I}d_i(u_i) \in \TI_n(R).\]
Choose $i\in I$ for which $u_i\neq 1$. Let $j\in I^c$. Without lost of generality we may assume that $i<j$. 

Given $\ell \geq 1$ and two elements $x,y$ in $\mbS_n(R)$, we denote by 
$[x,_\ell y]$ the left-normed iterated commutator $[[x, y], \ldots,y]$ with $\ell$ occurrences of $y$. 
We will show that $[e_{i,j}(1),_\ell d]\neq 1$ for all $\ell\in \Z_{> 0}$, contradicting the local nilpotency of $\barra{M}$. In fact, by induction on $\ell$ one can check that 
\[ [e_{i,j}(1),_\ell d]=e_{i,j}((1-u_i)^\ell).\]
Indeed using relations~\eqref{rel:commutatorsU} and~\eqref{conjdiag}, we find that 
\begin{eqnarray*}
[e_{i,j}((1-u_i)^\ell),d]& =& e_{i,j}((1-u_i)^\ell) d e_{i,j}(-(1-u_i)^\ell) d^{-1}\\
& = & e_{i,j}((1-u_i)^\ell)e_{i,j}(-u_i(1-u_i)^\ell)\\
& = & e_{i,j} ((1-u_i)^{\ell+1})
\end{eqnarray*}
Since $R$ is an integral domain and $u_i\neq 1$, it follows that $e_{i,j}((1-u_i)^\ell)$ is not the identity matrix (for otherwise $(1-u_i)^\ell$ would be zero). 
\end{proof}

\section{The group \texorpdfstring{$\PB_n(R)$}{PBn(R)} has \texorpdfstring{$R_\infty$}{R-infinity}} \label{sec:3} 
We start by proving Theorem~\ref{mainthm} in the specific case of $\PB_n(R)$, in order to illustrate our methods.
We start by recalling a theorem of Levchuk~\cite{LevchukOriginal} that describes all automorphisms of the groups  of unitriangular matrices for $n\geq 4$. This is used in \Cref{Sec:AutPBn} to characterize $\Aut(\PB_n(R))$. This description, in turn, is used in \Cref{Sec:PBn} to show that,  for the required rings $R$, every automorphism of $\PB_n(R)$ has infinitely many Reidemeister classes whenever $n\geq 4$.

\subsection{Automorphisms of upper triangular matrix groups}\label{sec:AutUn}
In~\cite{LevchukOriginal}, Levchuk describes automorphisms of groups of unitriangular matrices. 
Besides Levchuk's original paper, we refer the reader to \cite[Section~3]{Bn2} for a more detailed analysis and description of all types of automorphisms of unitriangular matrices. 

Being more precise, the elements of $\Aut(\mbU_n(R))$ are given in terms of a composition of specific types of automorphisms. Let us describe the three kinds of automorphisms which are the most important here.

The first automorphisms considered are inner automorphisms, i.e., maps of the form $\iota_{\mbf{u}}: \mbU_n(R) \to \mbU_n(R)$ (with $\mbf{u}\in \mbU_n(R)$) given by 
\[\iota_{\mbf{u}}(\mbf{h})=\mbf{u}\mbf{h}\mbf{u}^{-1}.\]
The group $\Inn(\mbU_n(R))$ of all inner automorphisms of $\mbU_n(R)$ is obviously a subgroup of $\Aut(\mbU_n(R))$. 
As illustrated by Relation~\eqref{rel:commutators}, $\mbU_n(R)$ is also invariant under conjugation by diagonal matrices. For this reason, inner automorphisms $\iota_d \in \Aut(\PB_n(R))$ with $d \in \mbD_n(R)$ also induce automorphisms of $\mbU_n(R)$, which by abuse of notation will also be denoted by $\iota_d$, despite not being an inner automorphism of $\mbU_n(R)$. {Note that  scalar matrices induce the identity automorphism. 

The next map is what we call the `flip' automorphism, denoted by $\tau \in \Aut(\mbU_n(R))$. Intuitively, one can think of $\tau$ as the transformation flipping the entries of the matrices in $\mbU_n(R)$ along the anti-diagonal and adjusting the signs of the entries in accordance with the commutator relations~\eqref{rel:commutatorsU}. More precisely, $\tau$ is given by the rule
\begin{equation}\label{tau} \tau(\eij(r)) = e_{n-j+1,n-i+1}((-1)^{j-i-1}r), \text{ for } \eij(r) \in \mbU_n(R). \end{equation}
One easily checks that $\tau^2 = \id$.

The last type of automorphisms which will play a significant role here comes from ring automorphisms of $R$. Given a ring automorphism $\alpha \in \anel{\Aut}(R)$, the induced map $\alpha_\ast : \mbU_n(R) \to \mbU_n(R)$ given by 
\[\alpha_\ast((a_{ij})) := (\alpha(a_{ij})) \quad \text{ for every matrix } \quad (a_{ij}) \in \mbU_n(R)\] 
is a group automorphism of $\mbU_n(R)$. With this in mind, we consider $\anel{\Aut}(R)$ as a subgroup of $\Aut(\mbU_n(R))$ in the obvious way and call it the \emph{subgroup of ring automorphisms} of $\mbU_n(R)$.

There are, of course, many more automorphisms in $\Aut(\mbU_n(R))$ not described here. Though for our purposes, the previous three will suffice in view of the next observation.

\begin{rmk}\label{rmk:ftrivial} Let $\mbU_n(R)'$ denote the commutator subgroup $[\mbU_n(R),\mbU_n(R)]$ of $\mbU_n(R)$. Recall that $\mbU_n(R)'$ is a characteristic subgroup of $\mbU_n(R)$.
Therefore, each automorphism $\varphi\in \Aut(\mbU_n(R))$ induces an automorphism $\overline{\varphi}$ on the quotient 
$\tfrac{\mbU_n(R)}{\mbU_n(R)'}$. Levchuk shows that the inner, flip, and ring automorphisms, together with certain additional automorphisms that induce the identity on the quotient $\tfrac{\mbU_n(R)}{\mbU_n(R)'}$, actually generate the full automorphism group $\Aut(\mbU_n(R))$. This will be made more precise in the next theorem. 
\end{rmk}

\begin{thm}[{Levchuk~\cite{LevchukOriginal}}] \label{thm:Levchuk}
Let $R$ be an integral domain, let $n \geq 4$, and suppose $R \neq \F_2$. Then any automorphism $\psi \in \Aut(\mbU_n(R))$ can be written as a product 
\[\psi = \iota_\mbf{u} \circ \iota_d \circ f \circ \tau^\veps \circ \alpha_\ast, \]
 where $\iota_\mbf{u} \in \Inn(\mbU_n(R))$, $\iota_d$ is the automorphism of $\mbU_n(R)$ induced by conjugation $\iota_d \in \Inn(\PB_n(R))$ with a diagonal matrix~$d$, $\veps \in \set{0,1}$, $\alpha_\ast$ is induced by a ring automorphism $\alpha \in \anel{\Aut}(R)$, and the remaining automorphism $f \in \Aut(\mbU_n(R))$ induces the identity on $\tfrac{\mbU_n(R)}{\mbU_n(R)'}$; see \cite{LevchukOriginal} or \cite[Section~3]{Bn2} for a full description of~$f$.
\end{thm}

In the next subsections, we assume $n\geq 4$ and $R\neq \F_2$ so that we can invoke Levchuk's result.

\subsection{Automorphisms of \texorpdfstring{$\PB_n(R)$}{PBn(R)}}\label{Sec:AutPBn}

A direct consequence of  \Cref{thm:Levchuk} is the following.
\begin{cor}\label{cor:dec} 
For $n\geq 4$ and $R$ an integral domain different from $\F_2$, 
every automorphism $\varphi\in \Aut(\PB_n(R))$ can be decomposed as $\varphi=\iota \circ \psi$, where $\iota\in \Inn(\PB_n(R))$  and 
\[\psi\bigm|_{\mbU_n(R)}=f \circ \tau^\veps \circ \alpha_\ast,\]
where $f$, $\tau^\veps$, and $\alpha_\ast$ are as in the statement of \Cref{thm:Levchuk}.
\end{cor}
\begin{proof}
Let $\varphi\in \Aut(\PB_n(R))$. By \Cref{thm:Levchuk} we know that $\varphi\bigm|_{\mbU_n(R)}$ is of the form $\iota \circ f \circ \tau^\veps \circ \alpha_\ast,$ where $f$, $\tau^\veps$, and $\alpha_\ast$ are as in the statement of \Cref{thm:Levchuk} and $\iota= \iota_\mbf{u} \circ \iota_d$,  
 where $\iota_\mbf{u} \in \Inn(\mbU_n(R))$, and $\iota_d$ is the automorphism induced by $\iota_d \in \Inn(\PB_n(R))$ for a diagonal matrix~$d$. Note that $\iota_{\mbf{u}d} \in \Inn(\PB_n(R))$ restricted to $\mbU_n(R)$ coincides with $\iota$.
 Now take $\psi=
\iota_{\mbf{u}d}^{-1} \circ \varphi \in \Aut(\PB_n(R))$. Then $\varphi = \iota \circ \psi$ and 
\[\psi\bigm|_{\mbU_n(R)}= (\iota^{-1} \circ \varphi)\bigm|_{\mbU_n(R)} =f \circ \tau^\veps \circ \alpha_\ast.\]
\end{proof}

In the sequel, we use this description of elements of $\Aut(\PB_n(R))$ to show that $\PB_n(R)$ has property~$\Ri$.

\subsection{Proof of \Cref{mainthm} for \texorpdfstring{$\PB_n(R)$}{PBn(R)}}\label{Sec:PBn}
In this section, let $R$ be either 
\begin{itemize}
    \item an integral domain of characteristic $0$ such that $\PB_n(R)$ is finitely generated, or
    \item of the form 
    $R=\F_q[t,f_1(t)^{-1}, \ldots, f_{\ell-1}(t)^{-1}]$ for a prime power~$q$, an integer~$\ell \geq 2$ and fixed pairwise coprime monic irreducible polynomials  $f_1(t)$, ..., $f_{\ell}(t) \in \F_q[t]$. 
\end{itemize}
The fact that $\PB_n(R)$ is finitely generated in the second case above is well-known; see, e.g., \cite[Theorem~1.3]{YuriSoluble}. 

Given an automorphism $\varphi \in \Aut(\PB_n(R))$, we want to show that $R(\varphi)=\infty$. By \Cref{cor:dec}, the automorphism $\varphi$ can be decomposed as $\varphi=\iota \circ \psi$, where $\iota \in \Inn(\PB_n(R))$ and $\psi\mid_{\mbU_n(R)}=f \circ \tau^\veps \circ \alpha_\ast$ (where we use the notation of Levchuk's decomposition; cf. \Cref{thm:Levchuk}). 
Since Reidemeister numbers are invariant under composition with inner automorphisms, it is enough to show that $R(\psi)=\infty$; see \Cref{lem:ignoreinner}.

The group $\mbU_n(R)'$ is a characteristic subgroup of  $\PB_n(R)$, as $\mbU_n(R)'$ is characteristic in $\mbU_n(R)$, which is itself characteristic in $\PB_n(R)$ by \Cref{pps:char}. Hence one can consider the automorphism $\Psi$ induced by $\psi$ on the quotient $\tfrac{\PB_n(R)}{\mbU_n(R)'}$.
We are thus considering the short exact sequence
\[1 \to \mbU_n(R)' \into  \PB_n(R) \onto \frac{\PB_n(R)}{\mbU_n(R)'} \to 1.\]
By \Cref{lem:quotient}, if $R(\Psi)=\infty$, then $R(\psi)=\infty$.
Let us then verify that $R(\Psi)=\infty$. 

Since the quotient $\tfrac{\PB_n(R)}{\mbU_n(R)'}$ is finitely generated (as $\PB_n(R)$ is so) and residually finite (by \Cref{lem:metabelianresiduallyfinite})}, it suffices to show that $|\Fix(\Psi)|=\infty$ in view of Jabara's trick~(\Cref{lem:Jabara}). 
We consider the restriction \[\Psi'=\Psi\bigm|_{\tfrac{\mbU_n(R)}{\mbU_n(R)'}} \in \Aut\left(\tfrac{\mbU_n(R)}{\mbU_n(R)'}\right).\]
Since $\Fix(\Psi') \subseteq \Fix(\Psi)$, it suffices to prove that $\Psi'$ has infinitely many fixed points. 

Now recall that \[\psi\bigm|_{\mbU_n(R)}= f \circ \tau^\veps \circ \alpha_\ast\] where
\begin{enumerate}
    \item $f \in \Aut(\mbU_n(R))$ induces the identity automorphism on $\tfrac{\mbU_n(R)}{\mbU_n(R)'}$,
    \item $\tau$ is the flip automorphism, and $\veps \in \set{0,1}$, and 
    \item $\alpha_\ast$ is induced by a ring automorphism $\alpha \in \anel{\Aut}(R)$.
\end{enumerate} 
Since the automorphism $\Psi'$ coincides with the automorphism induced by $\psi$ on $\tfrac{\mbU_n(R)}{\mbU_n(R)'}$, it is given by \[\overline{\tau^\veps \circ \alpha_\ast},\] the automorphism induced by $\tau^\veps \circ \alpha_\ast$. We are thus left to show that $|\Fix\left(\overline{\tau^\veps \circ \alpha_\ast}\right)|=\infty$. 
Before proving this, we require the following observation.

\begin{lem}\label{lem:transcendental} 
Let $q$ be a power of a prime and $m$ a positive integer. Fix pairwise coprime monic irreducible elements  $f_1(t)$, ..., $f_{m}(t) \in \F_q[t]$. Then for 
\[R=\F_q[t] \, \text{ or } \, R=\F_q[t,f_1(t)^{-1}, \ldots, f_{m}(t)^{-1}]\]
there exists a non-unit element $x \in R$ 
that is transcendental over $\F_q$ and invariant by all ring automorphisms of~$R$.
\end{lem}
\begin{proof}
First, fix a polynomial $g(t) \in \F_q[t] \subseteq R$ that is irreducible over $\F_q$ and coprime with  $t$, $f_1(t)$, ..., $f_{m}(t)$. Note that, if $R = \F_q[t]$, then 
\[U(R) = U(\F_q[t]) = \{ a \mid a \in \F_q \setminus \{0\} \},\]
whereas for $R = \F_q[t,f_1(t)^{-1}, \ldots, f_{m}(t)^{-1}]$ one has 
\[U(R)= \{a\cdot f_1(t)^{\varepsilon_1}\cdot \ldots \cdot f_{m}(t)^{\varepsilon_m} \mid a \in \F_q\setminus\{0\}, \quad \varepsilon_1, \ldots, \varepsilon_m \in \Z\}.\] 
Either way, the chosen element $g(t)$ satisfies  
\[g(t) \notin U(R).\] 
{By \cite[Remark 3.10 (ii)]{MitraSankaran} we know that $\anel{\Aut}(R)$ is finite and so it makes sense to set}
\[x = \prod_{\sigma \in \anel{\Aut}(R)} \sigma\left(g(t)\right).\]
By construction $x \notin U(R)$ --- being a finite product of non-units --- and $\sigma(x)=x$ for all $\sigma \in \anel{\Aut}(R)$. Moreover, since $R \subsetneq \F_q(t)$, the element $x$ is a non-constant rational function over $\F_q$ and hence transcendental over~$\F_q$.
\end{proof}

We are now in a position to prove $|\Fix\left(\overline{\tau^\veps \circ \alpha_\ast}\right)|=\infty$ for all rings considered.

\begin{pps} \label{pps:flipalphaPBn} Let $R$ be as in the beginning of this section and $n  \geq 4$. Then, the induced automorphism $\overline{\tau^\veps \circ \alpha_\ast} \in \Aut(\mbU_n(R)/\mbU'_n(R))$ as above has infinitely many fixed points.
\end{pps}
\begin{proof}
If $R$ is an integral domain with $\carac(R)=0$, it contains a copy of $\Z$, and any of its ring automorphisms $\alpha$ must fix all elements of $\Z$. If $R$ is of the form $R=\F_q[t,f_1(t)^{-1}, \ldots, f_{\ell-1}(t)^{-1}]$, then \Cref{lem:transcendental} assures the existence of a non-unit transcendental element $x \in R$ that is fixed by any of its ring automorphisms~$\alpha$. In particular, $\alpha(x^n)=x^n$, for all $n \in \Z_{>0}$. 

We then uniformly consider the infinite subset of $R$ given by
\[S= \begin{cases}
    \Z, & \text{ in characteristic 0, }\\
    \{x^n \mid n \in \Z_{>0}\}, & \text{ otherwise}.
\end{cases}\]
In particular, all elements of $S$ are fixed by all ring automorphisms of $R$.

Consequently the induced group automorphism $\overline{\alpha}_\ast$ on 
\[\mbU_n(R)/\mbU_n(R)'=\left\{\left( \begin{smallmatrix} 1 & * &  &  &  \\  & 1 & * & & \\ & & \ddots & \ddots & \\ & & & 1 & * \\ & & & & 1 \\  \end{smallmatrix}  \right)\mbU_n(R)' \right\}\cong R^{n-1}\]
has infinitely many fixed points. 
This shows the case $\varepsilon =0$.

 It remains to check that the map $\overline{\tau \circ \alpha_\ast}$ also has infinitely many fixed points. For this, consider the infinite subset of $\frac{\mbU_n(R)}{\mbU_n(R)'}$ given by
\[\mathcal{E}=\left\{ \bekl{1,2}(s)\cdot\bekl{n-1,n}(s)\mid s\in S\right\} = \left\{\left( \begin{smallmatrix} 1 & s & & & & \\  & 1 & 0 & & & & \\ & &  1 & 0 & & & \\ & & & \ddots & \ddots & &  \\ & & & & 1 & 0 & \\ & & & & & 1 & s \\ & & & & & & 1 \\  \end{smallmatrix} \right)\mbU_n(R)': s\in S\right\},\] 
where we are denoting $\bekl{i,j}(s)=\ekl{i,j}(s)\mbU_n(R)'$. Since $n\geq 4$, 
\[\overline{\tau^\veps \circ \alpha_\ast}(\bekl{1,2}(s)\cdot\bekl{n-1,n}(s))=\bekl{n-1,n}(s)\cdot\bekl{1,2}(s)=\bekl{1,2}(s)\cdot\bekl{n-1,n}(s)\] 
for all $s\in S$; cf. Relations~\eqref{rel:commutatorsU}. 
Therefore $\mathcal{E} \subset \Fix\left(\overline{\tau^\veps \circ \alpha_\ast}\right)$ and we are done. 
\end{proof}

\begin{cor}
    For $R$ an integral domain as in \Cref{mainthm}, the group $\mbB_n(R) = \mb{S}^{\{1,\ldots,n\}}_n(R)$ has property~$R_\infty$ whenever it is finitely generated and $n\geq 4$.
\end{cor}

\begin{proof}
    This follows from the facts that $\PB_n(R)$ has $R_\infty$, which we have just proved, and that $\PB_n(R)$ is a characteristic quotient of $\mbB_n(R)$ when $R$ is an integral domain. 
\end{proof}

\section{The general case}\label{sec:4}

Throughout this section, we let $n\geq 4$ and fix a proper subset $I \subsetneq \{1,\dots, n\}$ satisfying the (NG)-condition \eqref{conditionI}. 
As at the beginning of \Cref{Sec:PBn}, we assume throughout that $R$ is as in the assumptions of \Cref{mainthm}. That is, $R$ is either
\begin{itemize}
    \item an integral domain of characteristic $0$ such that $\mbS_n(R)$ is finitely generated, or
    \item of the form 
    $R=\F_q[t,f_1(t)^{-1}, \ldots, f_{\ell-1}(t)^{-1}]$ for a prime power~$q$, an integer~$\ell \geq 2$ and fixed pairwise coprime monic irreducible elements  $f_1(t)$, ..., $f_{\ell}(t) \in \F_q[t]$. 
\end{itemize}

This section completes the proof of \Cref{mainthm} by proving it for the remaining groups $\mbS_n(R)$. 
Similarly to \Cref{sec:3}, we use Levchuk's description of automorphisms of $\mbU_n(R)$ to decompose automorphisms of $\mbS_n(R)$. We then use this decomposition and a version of \Cref{pps:flipalphaPBn} to show that the group $\mbS_n(R)$ has property~$R_\infty$.

\subsection{Automorphisms of \texorpdfstring{$\mbS_n(R)$}{SI(R)}}\label{Sec:AutSI}

Recall that by Theorem~\ref{thm:Levchuk} every automorphism of 
$\mbU_n(R)$ can be written as a product
\[\iota \circ f \circ \tau^\veps \circ \alpha_\ast.\]
Observe that the automorphism $\iota\in\Aut(\mbU_n(R))$ induced by an inner automorphism of $\PB_n(R)$ can be decomposed as
$\iota= \iota_\mbf{u} \circ \iota_d$, where $\mbf{u}\in\mbU_n(R) $ and
$d\in \mbD_n(R)$ is a diagonal matrix. In turn, we can write $d$ as a product 
$d=d^I d^{c}$ with $d^I\in\TI_n(R)$ and $d^{c}\in\TIc_n(R)$. That is, $d^I$ is the diagonal matrix obtained from $d$ by replacing all entries indexed by elements of $I^{c}$ by one, and vice-versa for $d^{c}$. In particular, $\iota_{\mbf{u}} \circ \iota_{d^I} \in \Inn(\mbS_n(R))$ and $\iota_{d^{c}}\in \Aut(\mbU_n(R))$.

Completely analogous to Corollary~\ref{cor:dec}, we now obtain the following.

\begin{cor}\label{cor:decS}
Every automorphism $\varphi\in \Aut(\mbS_n(R))$ can be decomposed as $\varphi=\iota \circ \psi$, where $\iota \in \Inn(\mbS_n(R))$ and 
\[\psi\bigm|_{\mbU_n(R)}=\iota_{d^c} \circ f \circ \tau^\veps \circ \alpha_\ast,\]
where $f$, $\tau^\veps$, and $\alpha_\ast$ are as in the statement of Theorem~\ref{thm:Levchuk}, and $\iota_{d^c}$ is the automorphism induced by the inner automorphism of $\PB_n(R)$ given by a matrix $d^c \in \TIc_n(R)$. 
\end{cor}

\subsection{Proof of Theorem~\ref{mainthm} for \texorpdfstring{$\mbS_n(R)$}{SI(R)}}\label{Sec:proof}
Let $\varphi \in \Aut(\mbS_n(R))$. We now show that $R(\varphi)=\infty$. 

Corollary~\ref{cor:decS} assures that $\varphi=\iota \circ \psi$, where $\iota \in \Inn(\mbS_n(R))$ and 
\[\psi\bigm|_{\mbU_n(R)}=\iota_{d^c} \circ f \circ \tau^\veps \circ \alpha_\ast,\]
where $f$, $\tau^\veps$, and $\alpha_\ast$ are as in the statement of Theorem~\ref{thm:Levchuk}, and $\iota_{d^c}$ is the automorphism induced by the inner automorphism of $\mbB_n(R)$ given by a matrix $d^c \in \TIc_n(R)$. 

\begin{lem}\label{lem:d}
Let $d^c \in \TIc(R)$ be as above. By definition, there are $u_i \in U(R)$ for each $i\in I^c$ such that  
\[d^c=\prod_{i\in I^c}d_i(u_i)=\prod_{i=1}^{n}d_i(u_i), \text{ where we set } u_i = 1 \text{ for each } i\in I.\]
In this case, the diagonal matrix $d^c_\ast \in \mbD_n(R)$ defined by  
\[d^{c}_\ast=d_1(u_2)d_2(u_1)d_{n-1}(u_n)d_n(u_{n-1})\]
constructed from $d^c$ belongs to $\TI_n(R)$. 
\end{lem}
\begin{proof}
Let us first check that $d_1(u_2)d_2(u_1)$ is an element of $\TI_n(R)$. The (NG)-condition~\eqref{conditionI} on $I$ assures that $|\{1,2\}\cap I|\geq 1$. Note that, since $d^c \in \TIc_n(R)$, we actually have $u_i = 1$ whenever $i$ \emph{does} belong to $I$, in contrast with the usual setting.

If both $1,2 \in I$, then $u_1=u_2=1$. Whence $d_1(u_2)d_2(u_1)$ is the identity matrix.

Suppose $\{1,2\}\cap I=\{i\}$ and $\{1,2\}\cap I^c=\{j\}$. In this case, we have $u_i=1$ and $u_j\in U(R)$. In particular, 
\[d_1(u_2)d_2(u_1)=d_i(u_j)d_j(u_i)=d_i(u_j) \in \TI_n(R).\]
Arguing similarly, one concludes that $d_{n-1}(u_n)d_n(u_{n-1})\in \TI_n(R)$. Hence, $d^{c}_\ast\in \TI_n(R)$.
\end{proof}

\begin{lem}\label{thetrick} 
Every automorphism $\varphi\in \Aut(\mbS_n(R))$ can be decomposed as $\varphi=\iota_0 \circ \psi_0$, where $\iota_0 \in \Inn(\mbS_n(R))$ and \[\psi_0\bigm|_{\mbU_n(R)}=\iota_{d^{c}_\ast}\circ\iota_{d^c} \circ f \circ \tau^\veps \circ \alpha_\ast\]
for some $d^c \in \TIc_n(R)$ and where $f$, $\tau^\veps$, and $\alpha_\ast$ are as in the statement of Theorem~\ref{thm:Levchuk}, and $d^{c}_\ast$ is the matrix constructed from $d^c$ as in Lemma~\ref{lem:d}.
\end{lem}
\begin{proof}
Let $\varphi = \iota \circ \psi$ be the decomposition as in Corollary~\ref{cor:decS}. Take $\psi_0 = \iota_{d^{c}_\ast} \circ \psi$. Then clearly  $\varphi=(\iota\circ \iota_{d^{c}_\ast}^{-1}) \circ \psi_0$, where 
\[\psi_0\bigm|_{\mbU_n(R)}=\iota_{d^{c}_\ast}\circ\iota_{d^c} \circ f \circ \tau^\veps \circ \alpha_\ast,\]
and since
$\iota_{d^{c}_\ast} \in \Inn(\TI_n(R))$, it follows that $\iota\circ \iota_{d^{c}_\ast}^{-1}\in \Inn(\mbS_n(R))$.
\end{proof}

By \Cref{thetrick} and the invariance of Reidemeister numbers by inner automorphisms (cf. \Cref{lem:ignoreinner}), it suffices to show $R(\psi_0)=\infty$. For this, we adapt the arguments of \Cref{Sec:PBn}: firstly, we consider the automorphism $\Psi_0$ induced on the quotient $\tfrac{\mbS_n(R)}{\mbU_n(R)'}$ by $\psi_0$, and obtain a short exact sequence 
\[1 \to \mbU_n(R)' \into  \mbS_n(R) \onto \frac{\mbS_n(R)}{\mbU_n(R)'} \to 1.\]

Again, \Cref{lem:quotient} ensures that $R(\Psi_0)=\infty$ implies $R(\psi_0)=\infty$.

Once more, {since $\tfrac{\mbS_n(R)}{\mbU_n(R)'}$ is finitely generated and residually finite by \Cref{lem:metabelianresiduallyfinite}}, Jabara's trick~(\Cref{lem:Jabara}) reduces the problem of showing $R(\Psi_0)=\infty$ to proving that $|\Fix(\Psi_0)|=\infty$. 

We consider the restriction 
\[\Psi'_0=\Psi_0\bigm|_{\tfrac{\mbU_n(R)}{\mbU_n(R)'}} \in \Aut\left(\tfrac{\mbU_n(R)}{\mbU_n(R)'}\right).\]
Since $\Fix(\Psi'_0) \subseteq \Fix(\Psi_0)$, it suffices to prove that $\Psi'_0$ has infinitely many fixed points.
\medskip

In the following, given an automorphism $\Phi \in \Aut(\mbS_n(R))$, denote by $\overline{\Phi}$ the automorphism induced on the characteristic subgroup $\tfrac{\mbU_n(R)}{\mbU_n(R)'}$ of the metabelian quotient $\tfrac{\mbS_n(R)}{\mbU_n(R)'}$. 

Since the automorphism $\Psi'_0$ is the automorphism induced by $\psi_0$ on $\tfrac{\mbU_n(R)}{\mbU_n(R)'}$, it is given by \[\overline{\iota_{d^{c}_\ast}\circ \iota_{d^{c}}\circ \tau^\veps \circ \alpha_\ast}.\] Thus, it suffices to prove that  $|\Fix\left(\overline{\iota_{d^{c}_\ast} \circ \iota_{d^c}\circ \tau^\veps \circ \alpha_\ast}\right)|=\infty$. This is the content of the next proposition.

\begin{pps} \label{pps:flipalpha} Let $R$ be as in the beginning of this section and $n  \geq 4$. Then, the automorphism $\overline{\iota_{d^{c}_\ast}\circ \iota_{d^c}\circ\tau^\veps \circ \alpha_\ast} \in \Aut(\mbU_n(R)/\mbU_n(R)')$ defined above has infinitely many fixed points.
\end{pps}
\begin{proof}
In \Cref{pps:flipalphaPBn} we defined the following infinite subset of $R$ fixed pointwise by all its automorphisms 
\[S= \begin{cases}
    \Z, & \text{ in characteristic 0, }\\
    \{x^n \mid n \in \Z_{>0}\}, & \text{ otherwise},
\end{cases}\]
where the element $x\in R$ in the case of positive characteristic is as in \Cref{lem:transcendental}. 
We then showed that the restriction of 
$\overline{\tau^\veps \circ \alpha_\ast}$ to the infinite subset
\[\mathcal{E}=\left\{ \bekl{1,2}(s)\cdot\bekl{n-1,n}(s)\mid s\in S\right\} \subset \frac{\mbU_n(R)}{\mbU_n(R)'}\] 
is the identity map. 
Thus, we are left to show that 
\[\overline{\iota_{d^{c}_\ast}\circ \iota_{d^c}}(w)=w, \text{ for all }w\in \mathcal{E}.\]

Let $r\in R$. From Relation~\eqref{conjdiag}, we obtain 
\[\iota_{d^c}(\ekl{1,2}(r))=\ekl{1,2}(u_1u_{2}^{-1}r),\]
while 
\[\iota_{d^{c}_\ast}(\ekl{1,2}(r))=\ekl{1,2}(u_2u_{1}^{-1}r).\]
This yields
\[\overline{\iota_{d^{c}_\ast}\circ\iota_{d^c}}(\bekl{1,2}(r))=\bekl{1,2}(u_2u_1^{-1}u_1u_2^{-1}r)=\bekl{1,2}(r).\]

Similarly, 
\[\overline{\iota_{d^{c}_\ast}\circ\iota_{d^c}}(\bekl{n-1,n}(r))=\bekl{n-1,n}(r).\]
Thus, $\overline{\iota_{d^{c}_\ast}\circ \iota_{d^c}}$ also induces the identity on $\mathcal{E}$, which finishes off the proof.
\end{proof}

\section{Separating groups \texorpdfstring{$\mbS_n(R)$}{SIn(R)} with \texorpdfstring{$R_\infty$}{R-infinity} by cohomological properties}\label{sec:separating}

\Cref{mainthm} yields the following corollaries which `enrich' the known families of amenable groups with property~$\Ri$. Note that \Cref{cohomologicalmainthm} is an immediate consequence of Theorems~\ref{corollary:Abels1} and~\ref{corollary:Abels2} stated below. 

\begin{thm} \label{corollary:Abels1}
For every $n \in \N_{\geq 4}$ and $S \subset \N$ a finite subset of prime numbers, there exists a soluble group $\Gamma(n,S)$ with the following properties: 
\begin{enumerate}
    \item $\Gamma(n,S)$ is finitely presented and has property~$R_\infty$;     
    \item Its finiteness length is $\phi(\Gamma(n,S)) = n-2$;
    \item Its virtual cohomological dimension is a finite number $d(n,S) := \vcd(\Gamma(n,S)) \in \N$. Moreover, $d(n,S) \to \infty$ as $n \to \infty$ or $|S| \to \infty$.
\end{enumerate}
In particular, there exist infinitely many quasi-isometry classes of finitely presented, amenable, non-polycyclic groups with property $R_\infty$.
\end{thm}

\begin{proof}
Let $n$ and $S$ be as in the statement and let 
\[\mb{S}^{\{2,\ldots,n-1\}}_n = \mbA_n = \left( \begin{matrix}
  1 & * & \cdots & \cdots & * \\
  0 & * & \ddots & & \vdots \\
  \vdots & \ddots & \ddots & \ddots & \vdots \\
  0 & & \ddots & * & * \\
  0 & \cdots & \cdots & 0 & 1
 \end{matrix} \right) \leq \GL_n\]
denote the $\Z$-group scheme of H.~Abels. Consider 
\[R_S = \{ x \in \Q \mid \text{ denominators of } x \text{ have prime factors in } S \}.\]

(We remark that $R_S$ is a ring of $S'$-integers --- i.e., a Dedekind domain of $S'$-arithmetic type --- in $\Q$, where $S'$ is the (finite) set of places containing all archimedean places together with the non-archimedean places induced by the $p$-adic valuations with $p\in S$.) 

We set $\Gamma(n,S) := \mbA_n(R_S)$, which is then an $S'$-arithmetic subgroup of $\mbA_n(\Q)$. 
Note that $\Gamma(n,S)$ is non-polycyclic as it contains non-finitely generated subgroups such as $\mbU_n(R_S)$. 

The topological properties for $\Gamma(n,S)$ listed in the statement are known. More explicitly, \cite[Proposition~4.10]{YuriSoluble} gives $\phi(\Gamma(n,S))=n-2$. In particular, $\Gamma(n,S)$ is always finitely presented because $n \geq 4$; cf. \Cref{sec:geometry}. 

The fact that $\Gamma(n,S) = \mbA_n(R_S)$ has property~$\Ri$ follows from \Cref{mainthm}. Indeed, $\mbA_n = \mbS_n$ with $I = \{2,\ldots,n-2\}$. In particular, $I$ trivially satisfies the (NG)-condition \eqref{conditionI}. Since $R_S$ has characteristic zero and $\mbA_n(R_S)$ is finitely generated, the claim follows.

Regarding {the virtual cohomological} dimension, the group $\Gamma(n,S)$ is finitely generated and linear over $\Q$, whence virtually torsion-free by Selberg's lemma and moreover $d(n,S):=\mathrm{vcd}(\Gamma(n,S))<\infty$ by a classic result of Serre~\cite[Th\'eor\`eme~5, page~128]{SerreCohomologie}. To see why $d(n,S)$ increases if we increase $n$ or if we enlarge the set of primes $S$, recall that $\Gamma(n,S) = \mbA_n(R_S)$ fits into the short exact sequence
\[\mbU_n(R_S) \into \Gamma(n,S) \onto \prod_{i=2}^{n-1} \sbgpdi(R_S).\]
The nilpotent part $\mbU_n(R_S)$ contains $\mbU_n(\Z)$ by construction, which has $\vcd(\mbU_n(\Z)) = \binom{n}{2}$. The diagonal part $\prod_{i=2}^{n-1} \sbgpdi(R_S)$ is isomorphic to $n-2$ copies of the group of units $U(R_S)$, which is (finitely generated) abelian with torsion-free rank $|S|$. 
Either way, \Cref{lem:vcdsubgroups} shows that increasing $n$ or $|S|$ forces $d(n,S) = \vcd(\Gamma(n,S))$ to admit a strictly increasing  subsequence (on $n$ or $|S|$, possibly both).  
\end{proof}

Another series of examples shows that combining amenability with infinite virtual cohomological dimension is possible.

\begin{thm} \label{corollary:Abels2}
For every $n \in \N_{\geq 4}$, $\ell \in \N_{\geq 3}$, and $p$ either zero or a prime number, there exists a soluble group $\Gamma(n,\ell,p)$ with the following properties: 
\begin{enumerate}
    \item $\Gamma(n,\ell,p)$ is finitely presented and has property~$R_\infty$; 
    \item If $p = 0$, then $\Gamma(n,\ell,p)$ is linear over $\R$, is virtually torsion-free, but has $\vcd(\Gamma(n,\ell,p)) = \infty$;
    \item If $p$ is a prime, then $\Gamma(n,\ell,p)$ has $p$-torsion, dimension $\vcd(\Gamma(n,\ell,p)) = \infty$, and finiteness length $\phi(\Gamma(n,\ell,p)) = \ell-1$;
\end{enumerate}
In particular, there exist infinitely many quasi-isometry classes of finitely presented amenable groups with property $R_\infty$ and infinite virtual cohomological dimension.
\end{thm}

\begin{proof}
Let $n \geq 4$ and $\ell \geq 3$ be given. 
We define different groups depending on the parameter $p$. In case $p=0$, we 
shall again take Abels' groups --- choose $I = \{2,\ldots,n-1\}$ so that $\mbS_n = \mbA_n$. Now let $R_{0,\ell}$ denote the following subring of $\R$ with $x \in \R$ transcendental over $\Q$. 
\[R_{0,\ell} = \Z\left[\frac{1}{\ell!},x,x^{-1},(x+1)^{-1},\ldots,(x+\ell)^{-1}\right].\]
In particular, $R_{0,\ell}$ is an integral domain of characteristic zero. We then set $\Gamma(n,\ell,0) = \mbS_n(R_{0,\ell}) = \mbA_n(R_{0,\ell}) \leq \GL_n(\R)$. 

In case $p > 0$ is a prime, for our choice of base ring we start with polynomials $\F_p[t]$ and choose $\ell - 1$ pairwise coprime monic irreducible elements of $\F_p[t]$, say $f_1(t)$, ..., $f_{\ell-1}(t)$. 
Then, form the ring
\[R_{p,\ell} = \F_p[t,f_1(t)^{-1}, \ldots, f_{\ell-1}(t)^{-1}],\]
which is the $S'$-arithmetic subring of $\F_p(t)$ with $S'$ consisting of the (non-archimedean) places induced by the $\infty$-valuation $v_\infty(\frac{g(t)}{h(t)})=\mathrm{deg}(h)-\mathrm{deg}(g)$ and by the $f_i$-valuations on $\F_p(t)$. To form the group, choose any $I \subset \{1, \ldots, n\}$ with exactly $n-1$ elements and set $\Gamma(n,\ell,p) = \mbS_n(R_{p,\ell}) \leq \GL_n(\F_p(t))$. 
Note that $\Gamma(n,\ell,p)\cong \PB_n(R_{p,\ell})$ by \Cref{lem:isowithPBn}. 

Let us check that our groups have the prescribed properties. 
If $p=0$ we argue that we may apply \cite[Theorem~1.2(3)]{YuriSoluble} to see that $\Gamma(n,\ell,0)$ is finitely presented, which states that Abels' group $\Gamma(n,\ell,0) = \mbA_n(R_{0,\ell})$ is finitely presented as long as the upper triangular subgroup $\left( \begin{smallmatrix} * & * \\ 0 & * \end{smallmatrix} \right) \leq \SL_2(R_{0,\ell})$ is so. 
By \cite[Lemma~3.2]{YuriSoluble}, the latter is true if and only if the group
\[\Aff(R_{0,\ell}) := \left( \begin{smallmatrix} * & * \\ 0 & 1 \end{smallmatrix} \right) \leq \GL_2(R_{0,\ell})\]
is finitely presented. But $\Aff(R_{0,\ell})$ is isomorphic to $R_{0,\ell} \rtimes U(R_{0,\ell})$, where $R_{0,\ell}$ is the underlying additive group of the ring and its units $U(R_{0,\ell})$ act on $R_{0,\ell}$ by multiplication. 

In \cite[Corollary on p.~19]{KrophollerMullaney}, Kropholler--Mullaney show that a certain metabelian group $G_\ell$ is of homotopical type $\FPn{\ell +1}$ and so a fortiori of type $\FPn{2}$. By a famous result of Bieri--Strebel \cite[Theorem~5.4]{BieriStrebel}, types $\FPn{2}$ and $\mathtt{F}_2$ (thus finite presentability; cf. \Cref{sec:geometry}) are equivalent in this class, so that $G_\ell$ is finitely presented. 
When looking at the construction of the group $G_\ell$ one can see that it is of the form $A_\ell \rtimes Q_\ell$, where $A_\ell =R_{0,\ell}$ and $Q_\ell$ is a subgroup of $U(R_{0,\ell})$ also acting on $A_\ell$ by multiplication.  It follows that 
$G_\ell =R_{0,\ell} \rtimes Q_\ell $ is a (normal) subgroup of $R_{0,\ell} \rtimes U(R_{0,\ell})$ and 
$(R_{0,\ell} \rtimes U(R_{0,\ell})) / G_\ell\cong U(R_{0,\ell})/Q_\ell$ is finitely generated and abelian (because $U(R_{0,\ell})$ itself is finitely generated abelian). As noted in \Cref{ex:AbelianFinfty}, a finitely generated abelian group has unbounded finiteness length. Therefore, we see that $\Aff(R_{0,\ell})\cong R_{0,\ell} \rtimes U(R_{0,\ell})$ is an extension of $G_\ell$ by $U(R_{0,\ell})/Q_\ell$, where the two groups have $\phi(G_\ell)\geq 2$ and $\phi(U(R_{0,\ell})/Q_\ell)\geq 2$. We thus conclude by \Cref{lem:finpropextensions} that $\Aff(R_{0,\ell}) \cong R_{0,\ell} \rtimes U(R_{0,\ell})$ --- thus also $\left( \begin{smallmatrix} * & * \\ 0 & * \end{smallmatrix} \right) \leq \SL_2(R_{0,\ell})$ by \cite[Lemma~3.2]{YuriSoluble} --- is of type $\Fn{2}$, i.e., finitely presented.

We may therefore invoke \cite[Theorem~1.2(3)]{YuriSoluble} 
as claimed, so that $\Gamma(n,\ell,0) = \mbA_n(R_{0,\ell})$ is finitely presented. 

For $p > 0$ a prime, our groups are $\Gamma(n,\ell,p) = \mbS_n(R_{p,\ell})$ with $|I|=n-1$, in which case they are isomorphic to $\PB_n(R_{p,\ell})$ by \Cref{lem:isowithPBn}. As mentioned earlier, this is an $S'$-arithmetic group (in positive characteristic) with $|S'| = 1+\ell-1=\ell$. The underlying algebraic group is $\PB_n = \mbU_n \rtimes \mbf{T}$, where $\mbf{T}$ is a standard maximal torus of the group scheme $\PGL_n$ (which is split over $\Z$), and $\mbU_n$ is the unipotent radical. We may therefore apply \cite[Theorem~1.3]{YuriSoluble} to conclude that $\phi(\Gamma(n,\ell,p)) = \ell-1$, as desired. As $\ell \geq 3$, the group $\Gamma(n,\ell,p)$ is thus always finitely presented.

With $\Gamma(n,\ell,p)$ being finitely presented (thus finitely generated) and $R_{p,\ell}$ being an integral domain, for $p \geq 0$, \Cref{mainthm} implies that $\Gamma(n,\ell,p) = \mbS_n(R_{p,\ell})$ has property~$R_\infty$ as we chose $I$ to clearly fulfil the (NG)-condition~\eqref{conditionI}.

It remains to check that $\vcd(\Gamma(n,p,\ell)) = \infty$. We start with $p=0$. Since 
\[U(R_{0,\ell}) = \set{\pm p_1^{a_1} \cdots p_k^{a_k} x^{b_1} (x+1)^{b_2} \cdots (x+\ell)^{b_{\ell+1}} \mid a_i,b_j \in \Z, p_m \in S},\]
where $S$ is the (finite) set of prime factors of $\ell!$, we see that the subgroup $\mbA_n(R_{0,\ell})^+ \leq \Gamma(n,\ell,0)$ with no negative signs on the diagonal has index $[\Gamma(n,\ell,0) : \mbA_n(R_{0,\ell})^+] = 2^{n-2} < \infty$. Moreover, $\mbA_n(R_{0,\ell})^+$ is torsion-free by construction. However, since $\Z[x] \subset R_{0,\ell}$ and $R_{0,\ell}$ embeds into $\mbA_n(R_{0,\ell})^+$ by means of the elementary subgroups $\gera{e_{i,j}(r) \mid r \in R_{0,\ell}}$, it follows that $\mbA_n(R_{0,\ell})^+$ contains free abelian subgroups of arbitrarily large rank. Thus $\mathrm{vcd}(\Gamma(n,\ell,0)) = \mathrm{cd}(\mbA_n(R_{0,\ell})^+) = \infty$ by \Cref{lem:vcdsubgroups}.

For $p>0$ a prime, we claim that any proper subgroup of finite index in $\Gamma(n,\ell,p) \cong \PB_n(R_{p,\ell})$ contains $p$-torsion and hence $\vcd(\Gamma(n,\ell,p))=\infty$. So let $H \leq \PB_n(R_{p,\ell})$ be such a subgroup. The intersection $H \cap \mbU_n(R_{p,\ell})$ thus also has finite index in $\mbU_n(R_{p,\ell})$. But the group $\mbU_n(R_{p,\ell})$ is locally $p$-finite (i.e., every finitely generated subgroup is a finite $p$-group), hence $H \cap \mbU_n(R_{p,\ell})$ contains $p$-torsion elements, so that $\vcd(\Gamma(n,\ell,p))=\infty$ (cf. \Cref{sec:geometry}).  
\end{proof}

\section{Automorphisms of special quotients of Abels' groups, and the proof of Theorem~\ref{diversemainthm}}\label{sec:Abels}
Throughout this section, we denote by $\mbA_4(R)$ the Abels' group $\mathbf{S}^{\{2,3\}}_4(R)$.

In \Cref{sec:fundamentallemma}, we consider the quotients
\[\ombA = \ombU \rtimes (\mathcal{D}_2(R) \times \mathcal{D}_3(R)),\]
where $\ombU= \mbU_4(R)/\mbU_4'(R)$. More precisely, we show that any automorphism of $\overline{\mbA_4(R)}$ induces a permutation of its superdiagonal entries, combined with the application of a homomorphism of $(R,+)$ to each.

 This is used to show the main result of \Cref{sec:notpenultimate}. Specifically, in this section we show that the quotient of $\mbA_4(R_f)$
\[\overline{\overline{\mbA_4(R_f)}}=\frac{\mbA_4(R_f)}{ Z(\mbU_4(R_f))}\]
has property $R_\infty$, where $R_f=\F_2[t,t^{-1},f(t)^{-1}]$ and $f(t) \in \mathbb{F}_2[t]\setminus \{t,\, t{-1}\}$ is an irreducible polynomial  
that is not self-reciprocal.

We then move to the proof of \Cref{diversemainthm}. Part~\eqref{diversemainthm2} is checked in \Cref{sec:qidAn} and uses the fact that $\overline{\overline{\mbA_4(R_f)}}$ has~$\Ri$. Part~\eqref{diversemainthm1} is verified in \Cref{sec:qidNonamenable}.

\subsection{Induced automorphisms on the metabelian top}\label{sec:fundamentallemma}

For a commutative ring $R$, let $\mbA_4(R)$ denote Abels' group $\mathbf{S}^{\{2,3\}}_4(R)$. Here, we consider the quotient 
\[\ombA = \ombU \rtimes (\mathcal{D}_2(R) \times \mathcal{D}_3(R)),\]
where $\ombU= \mbU_4(R)/\mbU_4'(R)$. We will show that any automorphism of $\overline{\mbA_4(R)}$  induces a permutation of the superdiagonal entries, combined with the application of a homomorphism of $(R,+)$ to each entry.

\medskip

\noindent{\bf Notation:} From now on, we will use boldface symbols $\bekl{i,j}(r)$ and $\bd{i}(u)$ to denote the natural projections of the elements $\ekl{i,j}(r)$ and $d_i(u)$ in the quotient group $\overline{\mbA_4(R)}= \mathbf{S}^{\{2,3\}}_4(R)/\mbU_4'(R)$. 

\medskip

\begin{lem}\label{lem:char}
    $\overline{\mbU_4(R)}$ is characteristic in $\overline{\mbA_4(R)}$.
\end{lem}
\begin{proof}
    We can proceed exactly as in the proof of \Cref{pps:char} and show that $\overline{\mbU_4(R)}$ is the Hirsch--Plotkin radical of $\overline{\mbA_4(R)}$. If $\overline{\mbU_4(R)}$ were not the Hirsch--Plotkin radical, then there would be a normal subgroup $\barra{M}$ of $\overline{\mbA_4(R)}$ strictly containing $\overline{\mbU_4(R)}$ which is also locally nilpotent. This subgroup must contain a non-trivial element of the form $\bd{}=\bd{2}(u_2)\bd{3}(u_3)$. If $u_3\neq 1 $, then by the same argument as in \Cref{pps:char} one can show that the iterated commutator $[\bekl{3,4}(1),_\ell \bd{}]=\bekl{3,4}((1-u_3)^\ell) \neq 1 $, which contradicts the local nilpotency of $\barra{M}$. If, $u_3=1$ then $u_2\neq 1$ and in this case we obtain that $[\bekl{2,3}(1),_\ell \bd{}]=\bekl{2,3}((1-u_2)^\ell) \neq 1 $ which again contradicts the local nilpotency of $\barra{M}$.
\end{proof}

First, fix an automorphism $\varphi$ of $\overline{\mbA_4(R)}= \overline{\mbU_4(R)} \rtimes (\mathcal{D}_2(R) \times \mathcal{D}_3(R))$. Since $\overline{\mbU_4(R)}$ is characteristic in $\overline{\mbA_4(R)}$, it follows that for each $r \in R$ and each pair $(i,j) \in \{(1,2),\, (2,3), \, (3,4)\}$, 
\[
\varphi(\bekl{i,j}(r)) = \bekl{1,2}(\Phi^{1,2}_{i,j}(r))\, \bekl{2,3}(\Phi^{2,3}_{i,j}(r))\, \bekl{3,4}(\Phi^{3,4}_{i,j}(r)).
\]

The next lemma shows that such $\varphi$ acts more simply on the superdiagonal by permuting its entries while applying {an additive automorphism} to each.

\begin{lem} \label{lem:fundamental} 
Let $\varphi$ be an automorphism of $\overline{\mbA_4(R)}$ as above. Then, there exist $\sigma \in \Sym(\{(1,2),(2,3),(3,4)\})$ and group automorphisms 
\[\Phi_{1,2}, \, \Phi_{2,3}, \, \Phi_{3,4}:(R,+) \to (R,+)\] such that
\[\varphi(\bekl{i,j}(r))=\bekl{\sigma(i,j)}(\Phi_{i,j}(r)),\]
for all $(i,j) \in \{(1,2),(2,3),(3,4)\}$ and $r\in R$.
\end{lem}
{\begin{rmk}
There is a slight abuse of notation in the statement of the above lemma. As $\sigma(i,j)$ is of the form 
$(k,l)$, we should read $\bekl{\sigma(i,j)}=\bekl{{(k,l)}}$ as being the same as $\bekl{k,l}$.
\end{rmk}}
\begin{proof}
First, consider the maps $\Phi_2, \Phi_3 \colon U(R) \to U(R)$ that satisfy for each $u \in U(R)$
\[ \varphi(\bd{2}(u)) = \bd{2}(\Phi_2(u))\bd{3}(\Phi_3(u))E(u)\]
where $E(u) \in \overline{\mbU_4(R)} = \mbU_4(R)/\mbU'_4(R)$ is the unipotent part of $\varphi(\bd{2}(u))$.
Fix $u \in U(R)\setminus \{1\}$. For all $r \in R$, by applying $\varphi$ to both sides of 
\[\mu(\bd{2}(u), \bekl{3,4}(r))= \bekl{3,4}(r),\]
we obtain 
\[\mu\left(\bd{2}(\Phi_2(u))\bd{3}(\Phi_3(u)), \, \prod_{(i,j)}\bekl{i,j}(\Phi^{i,j}_{3,4}(r)) \right)=\prod_{(i,j)}\bekl{i,j}(\Phi^{i,j}_{3,4}(r)),\]
where the product ranges over $(i,j) \in \{(1,2),(2,3),(3,4)\}$.
From this, we conclude 
\[\bekl{1,2}(\Phi_2(u)^{-1}\Phi^{1,2}_{3,4}(r)) \, \bekl{2,3}(\Phi_2(u)\Phi_3(u)^{-1}\Phi^{2,3}_{3,4}(r)) \, \bekl{3,4}(\Phi_3(u)\Phi^{3,4}_{3,4}(r))\]
equals
\[\bekl{1,2}(\Phi^{1,2}_{3,4}(r)) \, \bekl{2,3}(\Phi^{2,3}_{3,4}(r)) \, \bekl{3,4}(\Phi^{3,4}_{3,4}(r)).\] Thus, for each $r$, we must have 
\begin{equation}\label{eq:threefour}
    \begin{cases}
    \Phi_2(u)^{-1}\Phi^{1,2}_{3,4}(r) & = \Phi^{1,2}_{3,4}(r), \\
    \Phi_2(u)\Phi_3(u)^{-1}\Phi^{2,3}_{3,4}(r) & = \Phi^{2,3}_{3,4}(r), \\
    \Phi_3(u)\Phi^{3,4}_{3,4}(r) & = \Phi^{3,4}_{3,4}(r).
\end{cases}
\end{equation}
Since $\varphi$ is an automorphism, there exists $r \in R$ such that 
\[\Phi^{i,j}_{3,4}(r) \neq 0, \quad \text{  for some } (i,j) \in \{(1,2),(2,3),(3,4)\}.\]
In particular, the cases~\eqref{eq:threefour} imply
\begin{equation}\label{eq:threefour2}
    \begin{cases}
    \Phi_2(u) = 1, & \text{ if } (i,j)=(1,2), \\
    \Phi_2(u)= \Phi_3(u), & \text{ if } (i,j)=(2,3), \\
    \Phi_3(u)=1, & \text{ if } (i,j)=(3,4). 
\end{cases}
\end{equation}
We now show that for all $s \in R$, we must have $\Phi_{3,4}^{k,\ell}(s)=0$ whenever $(k,\ell) \neq (i,j)$, showing the main claim for $\varphi(\bekl{3,4}(r))$. 

Suppose on contrary that there exists $s \in R$ such that 
\[\Phi^{k,\ell}_{3,4}(s) \neq 0,\]
for some $(k,\ell)\neq (i,j)$ in $\{(1,2),(2,3),(3,4)\}$. Then note that this also implies the cases~\eqref{eq:threefour2} depending on $(k,\ell)$.

Now, suppose that $(i,j)=(1,2)$. Then, by equation~\eqref{eq:threefour2} we have $\Phi_2(u)=1$. Since $\varphi$ is an automorphism, it follows that $\Phi_3(u)\neq 1$ and, in particular, $\Phi_3(u) \neq \Phi_2(u)$. In turn, this implies that $(k,\ell)\notin \{(2,3), (3,4)\}$. However, by assumption $(k,\ell)\neq (1,2)$, hence this case cannot occur. The case $(i,j)=(3,4)$ follows from analogous arguments. 

If $(i,j)=(2,3)$, then by equation~\eqref{eq:threefour2} we have $\Phi_2(u)=\Phi_3(u)$. Since $\varphi$ is an automorphism, $\Phi_2(u)=\Phi_3(u)\neq 1$. Hence, $(k,\ell)\notin \{(1,2), (3,4)\}$, and we conclude that this case is also not possible. 
Consequently $\varphi(\bekl{3,4}(r))=\bekl{\sigma(3,4)}(\Phi^{\sigma(3,4)}_{3,4}(r))$ for some 
$\sigma(3,4) \in \{ (1,2),(2,3),(3,4)\}$. 

Following the same steps, but using the relations
\begin{align*}
    \mu(\bd{3}(u), \bekl{1,2}(r)) &= \bekl{1,2}(r), \\
    \mu(\bd{2}(u)\bd{3}(u), \bekl{2,3}(r)) &= \bekl{2,3}(r),
\end{align*}
we conclude respectively that 
\[\varphi(\bekl{1,2}(r))=\bekl{\sigma(1,2)}(\Phi^{\sigma(1,2)}_{1,2}(r)), \qquad \varphi(\bekl{2,3}(r))=\bekl{\sigma(2,3)}(\Phi^{\sigma(2,3)}_{2,3}(r))\]
for some $\sigma(1,2), \sigma(2,3) \in \{(1,2),(2,3),(3,4)\}$.

Finally, $\varphi$ being an automorphism forces $\{\sigma(1,2), \sigma(2,3), \sigma(3,4)\}=\{(1,2),(2,3),(3,4)\}$ and so $\sigma \in \Sym(\{(1,2),(2,3),(3,4)\})$. For simplicity, we then write simply $\Phi_{i,j}$ instead of $\Phi^{\sigma(i,j)}_{i,j}$.

The fact that each $\Phi_{i,j}$ is bijective follows immediately from the bijectivity of $\varphi$. It is a group homomorphism of the additive group $(R,+)$ because for any $r_1, r_2 \in R$, we have 
\begin{align*} \bekl{\sigma(i,j)}(\Phi_{i,j}(r_1)+\Phi_{i,j}(r_2)) & =\bekl{\sigma(i,j)}(\Phi_{i,j}(r_1))\bekl{\sigma(i,j)}(\Phi_{i,j}(r_2)) \\
& =\varphi(\bekl{i,j}(r_1))\varphi(\bekl{i,j}(r_2)) \\ 
& =\varphi(\bekl{i,j}(r_1+r_2))\\
&= \bekl{\sigma(i,j)}(\Phi_{i,j}(r_1+r_2)).
\end{align*}
\end{proof}

\subsection{A special quotient of an Abels' group with $R_\infty$} \label{sec:notpenultimate}

We now consider a larger quotient of $\mbA_4(R)$, namely the quotient:
\[ \overline{\overline{\mbA_4(R)}}=\frac{\mbA_4(R)}{ Z(\mbU_4(R))}.\]

\medskip

\noindent{\bf Notation:} From now on, we will use the symbols $\cekl{i,j}(r)$ and $\cd{i}(u)$ to denote the natural projections of the elements $\ekl{i,j}(r)$ and $d_i(u)$ in the quotient group $\overline{\overline{\mbA_4(R)}}$.  

\medskip
\begin{pps}\label{e23-is-fixed}
Let $\varphi\in \Aut( \overline{\overline{\mbA_4(R)}}) $, then there exists an automorphism $\Phi_{2,3}$
of the additive group $(R,+)$ such that 
\[ \varphi(\cekl{2,3}(r)) = \cekl{2,3}(\Phi_{2,3}(r)) \mbf{u}(r) ,\]
where $\mbf{u}(r)\in \frac{\mbU'_4(R)}{ Z(\mbU_4(R))}$
\end{pps}
\begin{proof} {In exactly the same way as in \Cref{lem:char} one can show that $\mbU_4(R)/Z(\mbU_4(R))$ is 
characteristic in $\overline{\overline{\mbA_4(R)}}$ and so also 
 $\mbU'_4(R)/Z(\mbU_4(R))$ is characteristic in $\overline{\overline{\mbA_4(R)}}
$.} Hence $\varphi$ induces an automorphism on 
$\overline{\overline{\mbA_4(R)}}/(\mbU'_4(R)/Z(\mbU_4(R)))$ which is naturally isomorphic to 
$\overline{\mbA_4(R)}$. Using \Cref{lem:fundamental} we find that 
there exist $\sigma \in \Sym(\{(1,2),(2,3),(3,4)\})$ and group automorphisms 
\[\Phi_{1,2}, \, \Phi_{2,3}, \, \Phi_{3,4}:(R,+) \to (R,+)\] 
and maps $\mbf{u}_{i,j}:R \to \mbU'_4(R)/Z(\mbU_4(R))$ such that
\[\varphi(\cekl{i,j}(r))=\cekl{\sigma(i,j)}(\Phi_{i,j}(r))\mbf{u}_{i,j}(r) ,\]
for all $(i,j) \in \{(1,2),(2,3),(3,4)\}$ and $r\in R$.

We have to show that $\sigma(2,3)=(2,3)$. Suppose that $\sigma(2,3)\neq (2,3)$.
So $\sigma(2,3)$ is either $(1,2)$ or $(3,4)$. Let $(i,j)\in \{ (1,2) , (3,4)\}$ be such that 
\[ \{ \sigma(i,j), \sigma(2,3)\}= \{ (1,2), (3,4)\}.\]
Note that $\left[ \cekl{2,3}(1) , \cekl{i,j}(1) \right] \neq 1$. However, using the fact that $\mbU'_4(R)/Z(\mbU_4)$ is central in  $\mbU_4(R)/Z(\mbU_4(R))$, we find that 
\begin{eqnarray*}
\lefteqn{\varphi[\cekl{2,3}(1) , \cekl{i,j}(1) ]}\\
& = & 
[ \cekl{\sigma(2,3)}(\Phi_{2,3}(1)) \mbf{u}_{2,3}(1), 
\cekl{\sigma(i,j)}(\Phi_{i,j}(1)) \mbf{u}_{i,j}(1)]\\
& = & [\cekl{\sigma(2,3)}(\Phi_{2,3}(1)),\cekl{\sigma(i,j)}(\Phi_{i,j}(1)) ]\\
& = & 1\mbox{ (since }[e_{1,2}(r), e_{3,4}(s)] =1, \; \forall r,s \in R).
\end{eqnarray*}
This contradicts the fact that $\varphi$ is an automorphism, which finishes the proof.
\end{proof}

For the rest of this section, we consider the ring $R_f=\F_2[t,t^{-1},f(t)^{-1}]$, where $f(t) \in \mathbb{F}_2[t]\setminus \{t, \, t^{-1}\}$ is an irreducible polynomial.

Consider Abels' group $\mbA_4(R_f)$. The main objective of this section is to show that the quotient 
\[\overline{\overline{\mbA_4(R_f)}}=\frac{\mbA_4(R_f)}{ Z(\mbU_4(R_f))}\]
has property $R_\infty$ for appropriate choices of $f$.

To this end, we begin with a lemma concerning the ring automorphisms of~$R_f$. 
Recall that the reciprocal $f_r(t)$ of a polynomial $f(t)\in \F_2[t]$ of degree $n$ is defined by $f_r(t) = t^n f(1/t)$, and that $f(t)$ is self-reciprocal if $f_r(t)=f(t)$. 
E.g.\ when $f(t)= 1 + t + t^3$, then $f_r(t) = 1+ t^2+ t^3$, so $f(t)$ is not self-reciprocal. Moreover, it is irreducible, and hence $f(t)$ satisfies the conditions of the following lemma.  
\begin{lem}\label{lem:AutRing} Let $f(t) \in \mathbb{F}_2[t]\setminus \{t,\, t^{-1}\}$ be an irreducible polynomial 
which is not self-reciprocal, and write $R_f = \F_2[t,t^{-1},f(t)^{-1}]$. Then $\anel{\Aut}(R_f)=\{\id\}$. 
\end{lem}
\begin{proof}
Let $\varphi \in \anel{\Aut}(R_f)$. Let us show that $\varphi$ is the identity map. 

Since $\varphi$ sends units to units, there exist integers $a,b,c,d$ satisfying     $ad - bc = \pm 1$ and such that 
\begin{align*}
    \varphi(t) & = t^af(t)^b, \\
    \varphi(f(t)) & = t^cf(t)^d.
\end{align*}

    As $f(t) \in \F_2[t]$ is irreducible and not self-reciprocal, it is of the form
    \[f(t)=1+a_1t+\cdots + a_{n-1}t^{n-1}+t^n,\]
    for some $n\in \Z_{\geq 3}$ and some $a_1, \dots, a_{n-1} \in \F_2$ with 
    \[1+a_1t+ a_2t^2 + \cdots + a_{n-1}t^{n-1}+t^n \neq 1 + a_{n-1}t +\cdots +a_2 t^{n-2} + a_1 t^{n-1} + t^n .\]  
    We have
    \begin{align*}
        \varphi(f(t)) & = \varphi(1+a_1t+\cdots +t^n) \\
        & = 1+a_1t^af(t)^b +\cdots + t^{an}f(t)^{bn}.
    \end{align*}
    When comparing both values for $\varphi(f(t))$, we obtain 
    \begin{equation}\label{eq:comparison}
        t^cf(t)^d = 1+a_1t^af(t)^b +\cdots + t^{an}f(t)^{bn}.
    \end{equation}

In \cite[Proof of Theorem 5.1(ii)]{Bn1} this situation was also studied and it was proven that the matrix 
\[ M = \begin{pmatrix} a & b \\c & d
\end{pmatrix}\in \GL_2(\Z)\]
always has an eigenvalue one. Here we have to prove more as we need to show that $M$ is the identity matrix. 
Like in \cite{Bn1}, we will treat this problem via a case distinction and we will often use the following easy observation: assume that $\gamma,\delta \in \Z$ are such that 
\begin{equation}\label{product is polynomial}
t^\gamma f(t)^\delta = p(t), \mbox{ for some polynomial }p(t)\in \F_2[t],
\end{equation}
then $\gamma \geq 0$ and $\delta \geq 0$. 

\medskip

Indeed, suppose that $\gamma<0 $, then \eqref{product is polynomial} can be written as $f(t)^\delta = t^{-\gamma} p(t)$.
If $\delta> 0$, this would imply that the polynomial $t$ divides $f(t)$ which is a contradiction, if $\delta \leq 0 $, we find that 
$1 = t^{-\gamma} f(t)^{-\delta} p(t)$ which is again a contradiction. So we must have that $\gamma\geq 0$.\\
If $\gamma \geq 0$ and $\delta < 0$, we find that $t^\gamma = f(t)^{-\delta} p(t)$, implying that the polynomial $f(t)$ divides $t$, which is again a contradiction. As a conclusion, we have that $\gamma \geq 0$ and $\delta \geq 0$.
 \medskip

\noindent {\bf Case 1. $\mathbf{a>0}$}\\
\noindent {\bf Case 1.1 $ \mathbf{a>0,\; b\geq 0}$}\\
In this case, the right hand side of \eqref{eq:comparison} is a polynomial and the observation above implies that $c\geq 0$ and $d\geq 0$.
As $t$ does not divide $1+a_1t^af(t)^b +\cdots + t^{an}f(t)^{bn}$, $c=0$. From $\pm 1 = ad - bc=ad $, we find that $a=d=1$, as $a,d\geq 0$. Thus equation \eqref{eq:comparison} becomes $f(t) = 1+a_1tf(t)^b +\cdots + t^{n}f(t)^{bn}$. By comparing the degrees of the polynomials on both side of this equation, we see that $b=0$. So $a=d=1$ and $b=c =0$, which implies that $\varphi(t) =t$ (and $\varphi(f(t))=f(t)$) which in turn implies that $\varphi=\id$. 

\medskip

\noindent{\bf Case 1.2 $ \mathbf{a>0,\; b<  0}$}\\
Take $\beta = -b >0$. We can rewrite equation \eqref{eq:comparison} to 
\[ t^c f(t)^{d+n \beta} = f(t)^{n\beta} + a_1 t^a f(t)^{(n-1) \beta} + \cdots + a_{n-1} t^{a(n-1)}f(t)^\beta + t^{an}.\]
As the right hand side of this equation is a polynomial, it follows that $c \geq 0$ and $d+ n\beta \geq 0$. As $t$ does not divide the right hand side of the equation, we find that  $c=0$ and as $f(t)$ also does not divide the right hand side we must have that $d= - n \beta$. It follows that 
$\pm 1 = a d - bc = ad = - a \beta n$, which is a contradiction since $n \geq 3$. So this case is impossible. 

\bigskip

\noindent {\bf Case 2. $\mathbf{a<0}$}\\
\noindent {\bf Case 2.1 $ \mathbf{a<0,\; b> 0}$}\\
Let $\alpha = -a >0$. Now equation \eqref{eq:comparison} can be written as 
\[ t^{c+\alpha n} f(t)^d = t^{\alpha n} + a_1 t^{\alpha(n-1)}f(t)^b + \cdots + f(t)^{bn}.\]
We find that $c+ \alpha n\geq 0 $ and $d\geq 0$. As the right hand side of the equation is not divisible by $f(t)$, we have that $d=0$ and as the right hand side is also not divisible by $t$ it follows that $c = -\alpha n$. Hence $\pm 1 = ad -bc = b \alpha n$ which is again a contradiction, so we can also exclude this case.

\medskip

\noindent {\bf Case 2.2 $ \mathbf{a<0,\; b<  0}$}\\
With $\alpha =-a $ and $\beta =-b $, equation \eqref{eq:comparison} now becomes:
\[ t^{c+\alpha n} f(t)^{d+ n \beta} = t^{\alpha n}f(t)^{\beta n} + a_1 t^{\alpha(n-1)} f(t)^{\beta(n-1)}+ \cdots + 1,\]
from which we get that $c+\alpha n \geq 0$ and $d + \beta n\geq 0$. Moreover, the right hand side of the equation is neither divisible by $t$ nor by $f(t)$, from which we get that $c= -\alpha n = an $ and $d = -\beta n= bn$. Hence $\pm 1 = ad - bc = abn - b an =0 $, which shows that also this situation cannot occur.

\medskip

\noindent {\bf Case 2.3 $ \mathbf{a<0,\; b=  0}$}\\
In this case, we see that $\pm 1 = ad - bc = ad $, so $a = -1$ and $d=\pm 1$.
The equation \eqref{eq:comparison} reduces to 
\[ t^{c+n} f(t)^d = t^n + a_1 t^{n-1} + \cdots + a_{n-1} t + 1 = f_r(t)\]
This implies that $d\geq 0$ (so $d=1$) and $c+n\geq 0$. However, $f_r(t)$ is not divisible by $t$ and so $c+n=0$. It follows that 
$f(t) = f_r(t)$, but we have excluded this case because $f(t)$ is not self-reciprocal. 

\bigskip

\noindent {\bf Case 3. $\mathbf{a=0}$}\\
As $\pm 1 = ad - bc= bc$, there are two subcases to consider: $b=1$ and $b=-1$. \\
\noindent {\bf Case 3.1 $ \mathbf{a=0,\; b=1}$}\\
The equation \eqref{eq:comparison} now reads as 
\[ t^c f(t)^d = 1 + a_1f(t) + \cdots + f(t)^n.\]
We find that $c\geq 0$ and $d \geq 0$. As neither $f(t)$ nor $t$ divide the right hand side of the equation above, we see that $c=d=0$. However, this implies $\pm 1 = ad -bc = 0$ which is again a contradiction.

\medskip

\noindent {\bf Case 3.2 $ \mathbf{a=0,\; b=-1}$}\\
The equation \eqref{eq:comparison} now reads as 
\[ t^c f(t)^{d+n} = f(t)^n + a_1f(t)^{n-1} + \cdots + 1.\]
Now we find that $c\geq 0$ and $d+n \geq0$. As neither $f(t)$ nor $t$ divide the right hand side of the equation, we must have $c=d=0$. Again this implies $\pm 1 = ad -bc = 0$ which allows us to also exclude this last situation. 
\end{proof}

We shall need some additional technical observations. 

\begin{lem}\label{lemma-auto}
Let $R$ be an integral domain such that any element of $R$ can be written as a sum of units.
Assume that $\varphi$ is an automorphism of the additive group of $R$ and $\Phi\in \Aut(U(R))$ is an automorphism of the multiplicative group of $R$. If  
\[\Phi(u)\varphi(1) = \varphi(u)\mbox{  for all $u \in U(R)$},\]
then $\Phi$ can be extended to a ring automorphism $\overline{\Phi}$ of $R$.
\end{lem}

\begin{proof}
Let $r\in R$. By assumption, we can write $r$ as a finite sum of units:
\[r = \sum_{i =1}^n u_i.\]

We define 
\[ \overline{\Phi}(r) = \sum_{i=1}^n \Phi(u_i). \]
We first have to show that $\overline{\Phi}$ is well defined.
So assume that 
\[r = \sum_{i=1}^n u_i  =  \sum_{j=1}^m v_j,\]
where the $u_i$ and $v_j$ are units in $R$. We have to show that 
\[ \sum_{i=1}^n \Phi(u_i)  =  \sum_{j=1}^m  \Phi(v_j).\]
We have that 
\begin{eqnarray*}
\left(\sum_{i=1}^n \Phi(u_i) - \sum_{j=1}^m\Phi(v_j)\right)\varphi(1) & = & 
\sum_{i=1}^n \Phi(u_i)\varphi(1) - \sum_{j=1}^m \Phi(v_j)\varphi(1)\\
& = & \sum_{i=1}^n \varphi(u_i) - \sum_{j=1}^m\varphi(v_j)\\
& = & \varphi\left(\sum_{i=1}^n u_i - \sum_{j=1}^m v_j \right)\\
& = & \varphi( r-r)=0.
\end{eqnarray*}
As $\varphi(1)\neq 0$ (since $\varphi$ is an automorphism of the additive group of $R$), it must hold $\sum_{i=1}^n \Phi(u_i) - \sum_{j=1}^m \Phi(v_j)=0$, showing that $\overline{\Phi}$ is well defined.
It is obvious that $\overline{\Phi}\bigm|_{U(R)}= \Phi$.
By definition, it is clear that $\overline{\Phi}$ is a morphism of the additive group of $R$.
To show that it is also multiplicative, consider any two elements $r_1, r_2 \in R$ expressed as finite sums of units
\[r_1 =  \sum_{i=1}^n u_i \quad \text{ and } \quad r_2  = \sum_{j=1}^m v_j.\]
Then $r_1r_2= \displaystyle \sum_{i=1}^n\sum_{j=1}^m u_i v_j$ and we see that 

\begin{align*}
    \overline{\Phi}(r_1r_2) & = \sum_{i=1}^n\sum_{j=1}^m \Phi(u_i v_j)\\
    & = \left(\sum_{i=1}^n \Phi(u_i)\right) \left( \sum_{j=1}^m \Phi(v_j)\right)\\
    & = \overline{\Phi}(r_1)\overline{\Phi}(r_2).
\end{align*}
Note that for a general element $r = \displaystyle\sum_{i=1}^n u_i$, we have 
\[ \varphi(r) = \sum_{i=1}^n \varphi(u_i)=
\sum_{i=1}^n \Phi(u_i)\varphi(1)= \overline{\Phi}(r) \varphi(1).\]
As $\varphi$ is an automorphism of the additive group of $R$, there exists a $z\in R$ with 
$\varphi(z)=1$, and from the above, we then see that $1 = \overline{\Phi}(z) \varphi(1)$, showing that $\varphi(1)$ is a unit. Again by the above, we find that 
\[ \forall r \in R:\; \overline{\Phi}(r) = \varphi(r) \varphi^{-1} (1)\]
from which we deduce that $\overline{\Phi}$ is an automorphism of the ring $R$. 
\end{proof}
Recall our choice of domain $R=R_f = \F_2[t,t^{-1},f(t)^{-1}]$ with $f(t) \neq \{t, \, t^{-1}\}$ irreducible and not self-reciprocal. 

\begin{pps}\label{pps:conditionsonPhi}
Assume that $\varphi$ is an automorphism of the additive group of $R_f$ and $\Phi\in \Aut(U(R_f))$ is an automorphism of the multiplicative group of $R_f$. Then:
\begin{itemize}
\item If  
$\Phi(u)\varphi(1) = \varphi(u)\mbox{  for all $u \in U(R_f)$},$
then $\Phi(u)=u$ for all $u\in U(R_f)$.
\item If $\Phi(u)\varphi(1) = \varphi(u^{-1})\mbox{  for all $u \in U(R_f)$},$
then $\Phi(u)=u^{-1}$ for all $u\in U(R_f)$.
\end{itemize}
\end{pps}
\begin{proof}
Any element $r\in R_f$ is of the form 
\[r = \sum_{k,\ell \in \Z}a_{k,\ell} t^k f(t)^\ell =  \sum_{\ell \in L} u_\ell\]
where $L$ is a finite subset of $\Z$ and the $u_\ell$ are units in $R_f$. Hence $R_f$ satisfies the conditions of \Cref{lemma-auto}. When the first condition is satisfied, then we deduce from 
Lemma~\ref{lemma-auto} that $\Phi$ can be extended to a ring automorphism $\overline{\Phi}$ of $R_f$.
From Lemma~\ref{lem:AutRing}, we get that $\overline{\Phi}$ is the identity map on $R_f$ and so a fortiori $\Phi(u)=u$ for all $u\in U(R_f)$.

When the second condition, i.e.\ $\Phi(u)\varphi(1) = \varphi(u^{-1})$ for all $u\in U(R_f)$, is satisfied, we consider the map $\Psi: U(R_f) \to U(R_f): u \mapsto \Phi(u^{-1})$. Since $U(R_f)$ is an abelian group, $\Psi \in \Aut(U(R_f))$. The map $\Psi$ satisfies:
\[ \Psi(u) \varphi(1) = \Phi(u^{-1}) \varphi(1) = \varphi(u),\mbox{  for all $u \in U(R_f)$}, \]
and hence from the first case we deduce that $\Psi(u)=u$, and so $\Phi(u)=u^{-1}$, for all $u\in U(R_f)$.

\end{proof}

\begin{thm}\label{thm:AfR}
The group $\overline{\overline{\mbA_4(R_f)}}$ has property $R_\infty$, where $R_f$ is as in Lemma~\ref{lem:AutRing}.
\end{thm}
\begin{proof}

Fix $\varphi \in \Aut(\overline{\overline{\mbA_4(R_f)}})$, let us show that $R(\varphi)=\infty$.

Let $\overline\varphi$ be the induced automorphism on $\overline{\mbA_4(R_f)}$. It is enough to show that
$R(\overline\varphi)=\infty.$

Lemma~\ref{lem:fundamental} ensures that there exist $\sigma \in \Sym(\{(1,2),(2,3),(3,4)\})$ and group automorphisms 
\[\Phi_{1,2}, \, \Phi_{2,3}, \, \Phi_{3,4}:(R_f,+) \to (R_f,+)\] such that
\[\overline{\varphi}(\bekl{i,j}(r))=\bekl{\sigma(i,j)}(\Phi_{i,j}(r)),\]
for all $(i,j) \in \{(1,2),(2,3),(3,4)\}$ and all $r\in R_f$.\\
Moreover, Proposition~\ref{e23-is-fixed} implies that $\sigma(2,3)= (2,3)$.
Hence there are only two possibilities for $\sigma$, either $\sigma$ is the identity or $\sigma$ is the transposition which swaps $(1,2)$ and $(3,4)$. We will treat these two cases separately.

Let $\Phi_2, \Phi_3, \Psi_2, \Psi_3 \colon U(R_f) \to U(R_f) $ be group homomorphisms such that for all $u \in U(R_f)$
\begin{align*}
    \overline{\varphi}(\bd{2}(u)) & = \bd{2}(\Phi_2(u)) \, \bd{3}(\Phi_3(u))E(u),\\
    \overline{\varphi}(\bd{3}(u)) & = \bd{2}(\Psi_2(u)) \, \bd{3}(\Psi_3(u))F(u),
\end{align*}
where $E(u), F(u) \in \overline{\mbU_4(R)} = \mbU_4(R)/\mbU'_4(R)$ are the unipotent parts of $\overline{\varphi}(\bd{2}(u))$ and $\overline{\varphi}(\bd{3}(u))$. 

\medskip

We have the following relations in $\overline{\mbA_4(R_f)}$ for all $r\in R_f$ and all $u\in U(R_f)$:
\begin{eqnarray*}
     \mu(\bd{2}(u) , \bekl{3,4}(r))  &=& \bekl{3,4}(r)\\
 \mu(\bd{2}(u) \bd{3}(u) , \bekl{2,3}(r) ) & =& \bekl{2,3}(r)\\
 \mu (\bd{3}(u) ,  \bekl{1,2}(r) ) & =&  \bekl{1,2}(r) 
\end{eqnarray*}
Applying $\overline{\varphi}$ to the above relations, leads to

\begin{equation}\label{eqn1} 
\mu(\bd{2}(\Phi_2(u))\bd{3}(\Phi_3(u)),  \bekl{\sigma(3,4)}(\Phi_{3,4}(r)) )
 =\bekl{\sigma(3,4)}(\Phi_{3,4}(r))\end{equation}
\begin{equation}\label{eqn2}\mu ( \bd{2}(\Phi_2(u)\Psi_2(u)) \bd{3}(\Phi_3(u)\Psi_3(u)),  \bekl{2,3}(\Phi_{2,3}(r)))
 = \bekl{2,3}(\Phi_{2,3}(r))\end{equation}
\begin{equation}\label{eqn3} \mu (\bd{2}(\Psi_2(u))\bd{3}(\Psi_3(u)),   \bekl{\sigma(1,2)}(\Phi_{1,2}(r)))   = 
\bekl{\sigma(1,2)}(\Phi_{1,2}(r))\end{equation}

\medskip

\noindent \textbf{Case 1: $\sigma$ is the identity}\\
In this case equation~\eqref{eqn1} becomes
\[ \bekl{3,4}(\Phi_3(u) \Phi_{3,4}(r)) =\bekl{3,4}(\Phi_{3,4}(r))\]
from which we get that  $\Phi_3(u)=1$ for all $u\in U(R_f)$.
Analogously, equation~\eqref{eqn3} leads to $ \Psi_2(u) = 1$ for all $u\in U(R_f)$.
Using that $\Phi_3(u)=\Psi_2(u)=1$, equation~\eqref{eqn2} leads to 
\[ \bekl{2,3} ( \Phi_2(u) \Psi_3(u)^{-1} \Phi_{2,3}(r))  = \bekl{2,3}(\Phi_{2,3}(r)), \]
resulting in $\Phi_2(u) = \Psi_3(u)$ for all $u\in U(R_f)$. As $\overline\varphi$ has to induce an automorphism on $\mbA_4(R_f)/\mbU_4(R_f)\cong U(R_f) \times U(R_f)$, $\Phi_2=\Psi_3$ has to be an automorphism of $U(R_f)$.

Now, we apply $\overline{\varphi}$ to both sides of the equation $\mu(\bd{2}(u),  \bekl{2,3}(1)) = \bekl{2,3}(u)$ to get
\[\bekl{2,3}(\Phi_2(u)\Phi_{2,3}(1)) =\bekl{2,3}(\Phi_{2,3}(u))  \]
which yields
\[ \Phi_2(u)\Phi_{2,3}(1)= \Phi_{2,3}(u).\]
From \Cref{pps:conditionsonPhi}, we find that $\Phi_2(u)=\Psi_3(u)=u$ for all $u\in U(R_f)$.
This implies that the induced automorphism on $\mbA_4(R_f)/\mbU_4(R_f)\cong U(R_f) \times U(R_f)$ is the identity map, which has infinite Reidemeister number (all points of $U(R_f) \times U(R_f)$ are fixed) and so also $R(\overline{\varphi})=\infty$, which in turn shows that 
$R(\varphi) =\infty$.

\medskip

\noindent \textbf{Case 2: $\sigma$ swaps $(1,2)$ and $(3,4)$}\\
In this case equations~\eqref{eqn1} and \eqref{eqn3} lead to $\Phi_2(u) = \Psi_3(u)=1$ and using this,
equation~\eqref{eqn2} gives that $\Phi_3(u) = \Psi_2(u)$. The fact that $\overline\varphi$ has to induce an automorphism on $\mbA_4(R_f)/\mbU_4(R_f)\cong U(R_f) \times U(R_f)$ now shows that  $\Phi_3=\Psi_2$ has to be an automorphism of $U(R_f)$. Again we apply $\overline{\varphi}$ to the equation $\mu(\bd{2}(u),  \bekl{2,3}(1))= \bekl{2,3}(u)$, which now results in 
\[\bekl{2,3}(\Phi_3(u^{-1})\Phi_{2,3}(1))=\bekl{2,3}(\Phi_{2,3}(u))  \]
and we find that 
\[ \Phi_3(u)\Phi_{2,3}(1)= \Phi_{2,3}(u^{-1}) \mbox{ for all }u \in U(R_f).\]
This time, \Cref{pps:conditionsonPhi}, shows that $\Phi_3(u)=\Psi_2(u) =u^{-1}$.
Using the isomorphism $\mbA_4(R_f)/\mbU_4(R_f)\cong U(R_f) \times U(R_f)$, we find that 
$\overline{\varphi}$ induces the automorphism 
\[ \overline{\overline{\varphi}}: U(R_f) \times U(R_f) \to U(R_f) \times U(R_f): (u,v) \mapsto (v^{-1},u^{-1}).\]
All elements of the form $(u,u^{-1})$ are fixed points of $\overline{\overline{\varphi}}$ and so $
R(\overline{\overline{\varphi}})=\infty$, from which we again obtain that $R(\varphi)=\infty$.
\end{proof}

\subsection{Quasi-isometric diversity in quotients of Abels' groups}\label{sec:qidAn}

Here we prove \Cref{diversemainthm}\eqref{diversemainthm2}, for which we need a crucial observation due to Minasyan--Osin--Witzel, restated below in the form we need.

\begin{thm}[{\cite[Proof of Corollary~5.3]{MinasyanOsinWitzel}}] \label{thm:MOW}
Let $G$ be a finitely presented group. Suppose its centre $Z(G)$ contains a subgroup of the form $E = \bigoplus_{j \in \N} E_j$, where each $E_j$ is a non-trivial (abelian) subgroup. Given $J \subseteq \N$, set $E_J = \gera{E_j \mid j \in J} \leq Z(G)$. Then the set 
\[\mc{S} = \{ G/E_J \mid J \subseteq \N \}\]
of central quotients of $G$ contains uncountably many pairwise non-quasi-isometric groups.
\end{thm}

\begin{proofof}{\Cref{diversemainthm}\eqref{diversemainthm2}}
The goal is to apply \Cref{thm:MOW} to a carefully chosen group $G$ for which all of its central quotients $G/E_J$ have property~$R_\infty$. 

Let us first produce $G$ and argue why it satisfies the hypotheses of \Cref{thm:MOW}. Take Abels' group 
\[G = \mb{S}_4^{\{2,3\}}(R_f) = \mbA_4(R_f) \leq \GL_4(\F_2(t)),\] 
where 
\[R_f = \F_2[t,t^{-1},f(t)^{-1}] \subset \F_2(t)\]
is chosen as back in \Cref{sec:notpenultimate} with $f(t) \in \F_2[t]$ an irreducible polynomial which is not self-reciprocal. Recall that the centre of $G$ is the `upper corner subgroup'
\[Z(G) = Z(\mbA_4(R_f)) = \gera{e_{1,4}(r) \mid r \in R_f} \cong (R_f,+),\]
which can be easily checked using the relations \eqref{rel:commutatorsU} and \eqref{rel:commutators} seen in \Cref{sec:allrelations}. 
Thus $Z(G)$ contains the infinite-dimensional vector space $E = \F_2[t] \cong \bigoplus_{j \in \N} \F_2$ (we are choosing the canonical basis $\{1,t,t^2, \ldots\}$), and we may form the central quotients $G/E_J$ as in \Cref{thm:MOW}.

To see why $G$ is finitely presented, note that $R_f$ is an $S$-arithmetic ring in positive characteristic with $|S| = 3$. Hence \cite[Theorem 1.3]{YuriSoluble} implies that the Borel subgroup $\left( \begin{smallmatrix} * & * \\ 0 & * \end{smallmatrix} \right) \leq \SL_2(R_f)$ has finiteness length $2$, and is thus finitely presented. Therefore finite presentability of Abels' group $G = \mbA_4(R_f)$ is a consequence of \cite[Theorem~1.2(3)]{YuriSoluble}.

\Cref{thm:MOW} thus applies, and the collection of quotients $\{G/E_J\}_{J \subseteq \N}$ contains uncountably many pairwise non-quasi-isometric groups. But back in \Cref{thm:AfR} in \Cref{sec:notpenultimate} we have verified that all of these groups have property~$R_\infty$. 
{Indeed, it is easy to check that $Z(G/Z(G))=1$,  hence $Z(G/E_J)= Z(G)/E_J$ and so $(G/E_J)/Z(G/E_J)\cong G/Z(G)$. So we can see that the metabelian-by-abelian quotient 
\[G/Z(G) = \left( \frac{\mbU_4(R_f)}{Z(\mbA_4(R_f))} \right) \rtimes (\mc{D}_2(R_f) \times \mc{D}_3(R_f)) = \barra{\barra{\mbA_4(R_f)}}\]
is actually a characteristic quotient of each group $G/E_J$ by modding out $Z(G/E_J)$. Since $G/Z(G)$ has property~$R_\infty$ by \Cref{thm:AfR}, then so does each group $G/E_J$, which finishes off the proof.}
\end{proofof}

\subsection{Quasi-isometric diversity in the non-amenable case} \label{sec:qidNonamenable}

Now we deduce \Cref{diversemainthm}\eqref{diversemainthm1}. Since the focus of our paper was on soluble matrix groups, and because 
the proof of \Cref{diversemainthm}\eqref{diversemainthm1} relies on several external ingredients, 
this section is unavoidably not fully self‑contained. 

Much like the case of Part~(ii), we need an important observation about quasi-isometric diversity. This time due to Kropholler--Leary--Soroko \cite{KLS}, who noted that Leary's remarkable constructions \cite{IanFP} produce uncountably many quasi-isometry classes.

\begin{thm}[{\cite[Corollary~1.3]{KLS}}] \label{thm:KLS}
Given $n \geq 4$ there exists an uncountable set $\Lambda$ and $n$-dimensional Poincar\'e duality groups $\{G_\lambda\}_{\lambda \in \Lambda}$ that are non-finitely presented, of type $\FPn{}$, and pairwise non-quasi-isometric. 
\end{thm}

\begin{proofof}{\Cref{diversemainthm}\eqref{diversemainthm1}}
Let $\{G_\lambda\}_{\lambda \in \Lambda}$ be the uncountably many pairwise non-quasi-isometric groups as in \Cref{thm:KLS}. Since these are Poincar\'e duality groups in dimension $n \geq 4$, one has $H^1(G_\lambda; \Z[G_\lambda]) = 0$; cf. \cite[Item~(c), p.~139]{Bieri}. By Specker's formula for the number of ends (see, for example, \cite[Theorem~1]{DunwoodyEnds}), each $G_\lambda$ has $1+\dim_{\F_2}H^1(G_\lambda; \F_2[G_\lambda])$ ends. But as $H^1(G_\lambda; \F_2[G_\lambda]) \cong H^1(G_\lambda; \Z[G_\lambda]) \otimes_{\Z} \F_2$ (see, e.g., \cite[Corollary~3.7]{Swan}), 
we conclude that such groups $G_\lambda$ are one-ended. 

Now form the free product $\Gamma_\lambda = \Z^2 \ast G_\lambda$. We argue that $\Gamma_{\lambda}$ and $\Gamma_{\mu}$ are quasi-isometric if and only if $\lambda = \mu$. Indeed, if $\lambda = \mu$ then the groups are actually equal. Conversely, if $\Gamma_\lambda$ and $\Gamma_\mu$ are quasi-isometric, then a well-known result of Papasoglu--Whyte \cite[Theorem~3.1]{PapasogluWhyte} informs us that the quasi-isometry classes of the factors $\Z^2$ and $G_\lambda$ of the free product $\Gamma_\lambda = \Z^2 \ast G_\lambda$ have to match the quasi-isometry classes of the factors $\Z^2$ and $G_\mu$ of $\Gamma_\mu = \Z^2 \ast G_\mu$. But $\Z^2$ is not quasi-isometric to any $G_\nu$, $\nu \in \Lambda$, for multiple reasons --- for example, the groups $G_\nu$ are not finitely presented. 
Hence $G_\lambda$ has to be quasi-isometric to $G_\mu$, which by \Cref{thm:KLS} is the case only when $\lambda=\mu$. Therefore the groups $\Gamma_\lambda$ with $\lambda \in \Lambda$ are pairwise non-quasi-isometric. 

By \Cref{lem:FP}, the duality groups $G_\lambda$ from \Cref{thm:KLS} are of type $\FPn{\infty}$ and have $\mathrm{cd}(G_\lambda)<\infty$, and the same holds for $\Z^2$ (\Cref{ex:AbelianFinfty}). Due to Mayer--Vietoris, free products inherit all properties $\FPn{n}$ \cite[Proposition~2.13(a)]{Bieri} 
and their cohomological dimension is also finite in case its factors are finite-dimensional \cite[Proposition~6.1]{Bieri}. Thus, again invoking \Cref{lem:FP}, $\Gamma_\lambda = \Z^2 \ast G_\lambda$ is (finitely generated and) of type $\FPn{}$. 

It is now clear that $\Gamma_\lambda$ is accessible (\cite[Definition on p.~449]{Dunwoody}), as it is a free product of one-ended groups, and by Stallings' theorem on ends it is also obvious that $\Gamma_\lambda$ is infinitely-ended. 
The fact that each $\Gamma_\lambda$ has property~$R_\infty$ now follows: we may apply both the criterion of Iveson--Martino--Sgobbi--Wong \cite[Theorem~5.1.18]{IvesonMartinoSgobbiWong2025} or of Fournier-Facio \cite[Corollary~8.1.6]{IvesonMartinoSgobbiWong2025}.
\end{proofof}

\section{Optimality of hypotheses in Theorem~\ref{mainthm}} \label{sec:reallastsection}

As noted in the introduction, the bound $n \geq 4$ in \Cref{mainthm} is sharp and the  `no gaps shape' of $I \subseteq \{1,\ldots,n\}$ is optimal. 
Here we indicate which groups without property~$R_\infty$ demonstrate this. 

\subsection{Some groups in case $n=3$}

Recall from \Cref{pps:SIfinitelygenerated} that if $R$ is an integral domain whose additive group $(R,+)$ is finitely generated, then all $\mbS_n(R)$ are finitely generated.

\begin{lem}\label{lem:dim3char0trivialunits}
    Let $R$ be an integral domain of characteristic zero such that $(R,+)$ is finitely generated. If the units of $R$ are $U(R) =\{-1,1\}$, then $\mb{S}_3^{\{2\}}(R)$ does not have property $R_\infty$. 
\end{lem}

\begin{proof}
    For simplicity, we write $\mb{S}_3^{\{2\}}(R)=\mbA_3(R)$. 
    One can check that 
    \[\varphi\begin{pmatrix}
        1 & x & z \\ 0 & u & y \\ 0 & 0 &1
    \end{pmatrix}= \begin{pmatrix}
        1 & 2x+uy & x^2+uxy-z \\ 0 & u & ux \\ 0 & 0 & 1
    \end{pmatrix}\]
    is a group automorphism. Let us show that $R(\varphi)<\infty$. 
    
    In fact, there exists a short exact sequence
    \[1 \to \mbU_3(R) \into  \mbA_3(R) \onto \mathcal{D}_2(R)  \to 1,\]
    and, since $\mbU_3(R)$ is a characteristic subgroup of $\mbA_3(R)$, we can consider the induced automorphisms $\varphi' \in \Aut(\mbU_3(R))$ and $\overline{\varphi} \in \Aut(\mathcal{D}_2(R))$. To show that $R(\varphi)<\infty$, it suffices to show that $R(\overline{\varphi})<\infty$ and $R(\iota_m \circ \varphi')<\infty$ for all inner automorphisms $\iota_m$ with $m \in \mbA_3(R)$ (\Cref{lem:heath}). 
    
    It is clear that $R(\overline{\varphi})<\infty$ because $\mathcal{D}_2(R)$ is isomorphic to $U(R) \cong \Z/2\Z$.
    Also, we know that every element $ m \in \mbA_3(R)$ can be written as $m=\mathbf{u}d$ with $\mathbf{u} \in \mbU_3(R)$ and $d \in \mathcal{D}_2(R)$. Thus, $\iota_m = \iota_{\mathbf{u}} \circ \iota_d$ and consequently 
    \[R(\iota_m \circ \varphi') = R(\iota_{\mathbf{u}} \circ (\iota_d \circ \varphi'))=R(\iota_d \circ \varphi'),\]
    by~\Cref{lem:ignoreinner}. For simplicity, denote $\psi_d= \iota_d \circ \varphi'$. Thus, it suffices to check that $R(\psi_d)< \infty$ for $d=d_2(1)$ and $d=d_2(-1)$. 

    We have the following explicit expressions for $\psi_{d_2(1)}$ and $\psi_{d_2(-1)}$:
 \[\psi_{d_2(1)}\begin{pmatrix}
        1 & x & z \\ 0 & 1 & y \\ 0 & 0 &1
    \end{pmatrix}= \begin{pmatrix}
        1 & 2x+y & x^2+xy-z \\ 0 & 1 & x \\ 0 & 0 & 1
    \end{pmatrix} \mbox{ and}\]
 \[\psi_{d_2(-1)}\begin{pmatrix}
        1 & x & z \\ 0 & 1 & y \\ 0 & 0 &1
    \end{pmatrix}= \begin{pmatrix}
        1 & -2x-y & x^2+xy-z \\ 0 & 1 & -x \\ 0 & 0 & 1
    \end{pmatrix}. \]
It is clear that $\Fix(\psi_{d_2(1)}) = \Fix(\psi_{d_2(-1)})=1$, from which we may conclude that both $R(\psi_{d_2(1)})< \infty$ and $R(\psi_{d_2(-1)})<\infty$ by Lemma~\ref{inverse-Jabara}, since $\mbU_3(R)$ is a finitely generated and torsion free nilpotent group.
\end{proof}

\begin{pps}\label{pps:dim3char0}
    Let $R$ be an integral domain of characteristic zero such that $(R,+)$ is finitely generated. Then $\mb{S}_3^{\{2\}}(R)$ does not have property $R_\infty$. 
\end{pps}

\begin{proof}
    Due to the previous lemma, we can consider the case where $U(R) \neq \{ 1, -1\}$ and so we can fix a unit $v\in U(R) \setminus \{1,-1\}$.
    Having fixed this unit we consider the automorphism $\varphi\in\Aut(\mb{S}_3^{\{2\}}(R))$ which is given by 
    \[\varphi\begin{pmatrix}
         1 & x & z \\ 0 & u & y \\ 0 & 0 &1
    \end{pmatrix}= \begin{pmatrix}
            1 & -\frac{y}{u} & -\frac{ v x y }{u} + vz  \\ 0 & \frac{1}{u} & \frac{ vx }{u}\\ 0 & 0 &1
    \end{pmatrix}.\]
 As in the previous proposition, we consider the short exact sequence
    \[1 \to \mbU_3(R) \into  \mbA_3(R) \onto \mathcal{D}_2(R)  \to 1,\]
    where $\mbU_3(R)$ is a characteristic subgroup of $\mb{S}_3^{\{2\}}(R)$. We can consider the induced automorphisms $\varphi' \in \Aut(\mbU_3(R))$ and $\overline{\varphi} \in \Aut(\mathcal{D}_2(R))$. In order to show that $R(\varphi)<\infty$, it again suffices to show that $R(\overline{\varphi})<\infty$ and $R(\iota_{d(u)} \circ \varphi')<\infty$ for all inner automorphisms $\iota_{d(u)} $ with 
    \[ d(u) = \begin{pmatrix}
         1 & 0 & 0 \\ 0 & u & 0 \\ 0 & 0 &1
    \end{pmatrix}.\]
The automorphism $\bar\varphi:\mathcal{D}_2(R)\cong U(R) \to \mathcal{D}_2(R)\cong U(R)$ is sending each element to its inverse and hence $R(\varphi)=[U(R): U(R)^2]<\infty$ since $U(R),\cdot$ is a finitely generated abelian group. 

We have the following explicit description for $\iota_{d(u)} \circ \varphi'$:
\[ \iota_{d(u)} \circ \varphi' \begin{pmatrix}
        1 & x & z \\ 0 & 1 & y \\ 0 & 0 &1
    \end{pmatrix}= \begin{pmatrix}
            1 & -\frac{y}{u} & -v x y + vz  \\ 0 & 1& uvx\\ 0 & 0 &1
    \end{pmatrix}.\]
One easily checks that $\Fix( \iota_{d(u)} \circ \varphi') =1$, and so Lemma~\ref{inverse-Jabara} allows us to conclude that $R(\iota_{d(u)} \circ \varphi') <\infty$ for all $u\in U(R)$, which finishes the proof. 
\end{proof}

\subsection{Proof of Proposition~\ref{obs:mainthmisoptiomal}}

As mentioned in the introduction, \Cref{obs:mainthmisoptiomal} is essentially known, but we provide a proof for completeness.

\underline{Part~(i):} When $I = \varnothing$ we have that $\mb{S}_n^{I}(R) = \mbU_n(R)$ by design. Now choose $R = \mathcal{O}_{\K}$ to be a ring of integers of an algebraic number field $\K$ such that $U(\mathcal{O}_{\K})$ is infinite --- for instance, $\mathcal{O}_{\K} = \Z[\sqrt{2}]$. Regardless of $n \geq 2$, Nasybullov constructs in \cite[(Proof of) Proposition~8]{TimurUniTri} an automorphism $\phee$ of $\mbU_n(\mathcal{O}_{\K})$ that has $R(\phee) < \infty$. 

\underline{Part~(ii):} For $n=2$ we may pick $I = \{1\}$ or even $I =\{1,2\}$, both of which obviously fulfil the (NG)-condition~\eqref{conditionI}. Recall from \Cref{lem:isowithPBn} that $\mb{S}^{\{1\}}_2(R) \cong \PB_2(R)$. Now choosing $R = \F_q[t]$, it is already known by \cite[Proposition~6.1]{Bn2} that $\mbS_2(\F_q[t])$ has an automorphism whose Reidemeister number is finite and divisible by $q-1$. And in case $n=3$, it is a consequence of \Cref{pps:dim3char0} that Abels' groups $\mb{S}_3^{\{2\}}(\mathcal{O}_{\K}) = \mbA_3(\mathcal{O}_{\K})$ over rings of integers of algebraic number fields also do not have property~$R_\infty$. \qed

\medskip

Note that all groups in the proof above, with the exception of $\PB_2(\F_q[t])$ and $\mbB_2(\F_q[t])$, are finitely generated. At the moment we do not know whether the finitely generated groups $\PB_2(\F_q[t,t^{-1}])$ and $\mbB_2(\F_q[t,t^{-1}])$ have property~$R_\infty$. It is interesting to note that the (finitely presented) groups $\PB_2(\F_2[t,t^{-1},f(t)^{-1})])$ and $\mbB_2(\F_q[t,t^{-1},f(t)^{-1}])$, with $f(t) \in \F_2[t] \setminus \{t,t-1\}$ irreducible, actually do have~$R_\infty$; cf. \cite[Theorem~5.1]{Bn1}.

\section*{Acknowledgments}

This work was initiated and supported through the ``OWRF / Research in Pairs'' programme of the Mathematisches Forschungsinstitut Oberwolfach (MFO), Germany. PMLA and YSR gratefully thank the direction and all staff of the MFO for the support and hospitality. KD and PMLA were also partially supported by the Methusalem grant METH/21/03 -- long term structural funding of the Flemish Government. KD was also supported by an FWO sabbatical bench fee (K803525N). YSR was partially supported by the Deutsche
Forschungsgemeinschaft (DFG, German Research Foundation), 314838170, GRK 2297 MathCoRe. PMLA and YSR thank the KU Leuven, Campus Kulak Kortrijk, for their hospitality in June 2025. KD thanks the University of Lincoln for their hospitality in November 2025 and January 2026. YSR thanks the University of Southampton for their hospitality in June 2022, and particularly Ian J. Leary, Peter Kropholler, Kevin Li, Armando Martino, and Ashot Minasyan for stimulating mathematical discussions.

 \def\cprime{$'$} \def\cprime{$'$}
 \providecommand{\bysame}{\leavevmode\hbox to3em{\hrulefill}\thinspace}
 \providecommand{\MR}{\relax\ifhmode\unskip\space\fi MR }

\printbibliography

\end{document}